\documentclass[11pt]{article}

\usepackage{graphicx}
\usepackage{algorithm}
\usepackage{amsmath}
\usepackage{amsfonts}
\usepackage{amssymb}
\usepackage{amsthm}
\usepackage{algorithmic}
\usepackage{subfig}
\usepackage{tabularx}
\usepackage{pdflscape}
\usepackage{stfloats}
\usepackage{epsfig}
\usepackage{placeins}
\usepackage{bbm}

\usepackage[utf8]{inputenc} 
\usepackage{array} 
\usepackage{paralist} 
\usepackage{verbatim} 

\usepackage{array}
\usepackage{longtable}
\usepackage{hyperref}

\usepackage{amsmath,amsfonts,latexsym,amssymb}

\newcommand{\banachb}{Y}

\def\K{\mathbf{K}}
\DeclareMathOperator{\degree}{deg}

\newcommand{\Sb}{\mathcal{S}}

\newcommand{\Sym}[1]{\mathrm{Sy}}

\newcommand{\Com}{\mathbb{C}}

\newcommand{\Vset}{v}

\newcommand{\G}{\mathrm{pv}}

\newcommand{\Lo}{\mathrm{pq}}
\newcommand{\powb}[2]{{\left({#1}\right)}^{#2}}	
\newcommand{\Vx}{V^x}
\newcommand{\Vy}{V^y}
\newcommand{\Vc}{V^c}

\usepackage{xargs}                      % Use more than one optional parameter in a new commands
\usepackage[pdftex,dvipsnames]{xcolor}  % Coloured text etc.

\usepackage[colorinlistoftodos,bordercolor=orange,backgroundcolor=orange!20,linecolor=orange,textsize=scriptsize]{todonotes}
\newcommandx{\reply}[2][1=]{\todo[linecolor=red,backgroundcolor=red!25,bordercolor=red,#1]{#2}}
\usepackage{amsmath}

\newcommand{\PutDelim}[4]{%
  \ifcase#1 #3#2#4 \or \bigl#3#2\bigr#4 \or \Bigl#3#2\Bigr#4%
    \or \biggl#3#2\biggr#4 \or \Biggl#3#2\Biggr#4%
  \else \left#3#2\right#4\fi}

\providecommand{\br}[1]{\left\{ #1 \right\}}

\newcommand{\R}{{\mathbb R}}

\newcommand{\N}{{\mathbb N}}

\newcommand{\neww}[1]{}

\newcommand{\coloneq}{\mathrel{\mathop:}=}

\newtheorem{assumption}{Assumption}{\bf}{} %\it
\newtheorem{proposition}{Proposition}{\bf}{} %\it
\newtheorem{definition}{Definition}{\bf}{}
\newtheorem{lemma}{Lemma}{\bf}{}
\newtheorem{theorem}{Theorem}{\bf}{}
\newtheorem{corollary}{Corollary}{\bf}{}
{\bf}{}
{\bf}{}

\begin{document}
\title{Hybrid Methods in Solving Alternating-Current Optimal Power Flows}

\author{Alan C. Liddell, Jie Liu, Jakub Mare\v{c}ek, Martin Tak\'a\v{c}% <-this % stops a space
\thanks{
A. C. Liddell is with the University of Notre Dame, South Bend, IN, USA. 
J. Liu and M. Tak\'a\v{c} are with Lehigh University, Bethlehem, PA, USA. 
J. Mare\v{c}ek is with IBM Research, Dublin, Ireland;
 e-mail: \url{jakub.marecek@ie.ibm.com}.}% <-this % stops a space
}

\maketitle

\begin{abstract}
Many steady-state problems in power systems, including rectangular power-voltage formulations of optimal power flows in the alternating-current model (ACOPF),
can be cast as polynomial optimisation problems (POP).
For a POP, one can derive strong convex relaxations, or rather hierarchies of ever stronger, but ever larger relaxations. 
We study means of switching from solving the convex relaxation 
to Newton method working on a non-convex augmented Lagrangian of the POP.
%In some cases, readily available second-order methods for solving the convex relaxations 
%fail to perform even a single iteration within reasonable run-times.
%First-order methods, which have much lower per-iteration computational and memory requirements, 
%can be applied, but require many more iterations than second-order methods to converge within the same accuracy.
\end{abstract}

\section{Introduction}

The alternating-current optimal power flow problem (ACOPF) is one of the best known
 non-convex non-linear optimisation problems, studied extensively since 1960s \cite{Low2014a,Low2014b,Panciatici,Ghaddar2015}.
Early work focused on applications of Newton method to the non-convex problem,
  which produced exceptionally fast routines, 
  albeit without any guarantees as to their global convergence. 
More recently, Lavaei and Low \cite{lavaei2012zero} have shown that a semidefinite 
  programming (SDP) relaxation produces global optima in some cases.
Ghaddar et al. \cite{Ghaddar2015} have shown that the SDP relaxation can be strengthened
  iteratively such that the hierarchy of relaxations converges to the global optimum of the non-convex problem, 
  asymptotically, under mild conditions, 
  albeit at a considerable computational cost.
It has not been clear, however, how to combine the two approaches. 
%See \cite{Low2014a,Low2014b,Panciatici} for further references.  

On one hand, Newton method's local convergence can lead to a variety of poor outcomes.
When one starts from an initial point outside of a neighbourhood of a stationary point, Newton method may diverge
and produce no feasible solution. %Kantorovich
Even within the neighbourhood, where Newton method converges, the stationary point may turn out to be very far from
the global optimum. 
For an illustration, see Figure \ref{fig:motivation},  which we discussed in more detail in Section \ref{sec:practice}.
This behaviour is inherent in the non-convexity of the problem.

On the other hand, solving the strengthened SDP relaxations \cite{Ghaddar2015} is challenging, computationally.
Leading second-order methods for solving the SDP relaxations, such as SeDuMi \cite{sturm1999using}, %SDPA \cite{yamashita2003implementation}, and SDPT3 \cite{toh1999sdpt3}, 
often converge within dozens of iterations on SDP relaxations of even the largest available instances available, 
but the absolute run-time and memory requirements of a single iteration may 
be prohibitively large.
Alternatively, one may employ first order methods \cite{MarecekTakac2015,7403152} 
whose memory requirements and per-iteration run-times are trivial, 
but whose rates of convergence are sub-linear.
Either way, as one progresses in the hierarchy, the run-time to reach acceptable accuracy grows fast. % becomes challenging.

To address the challenge, we introduce novel means of combining solvers working on the
 convexification and solvers working on the non-convex problem.
We employ a first-order method in solving the convexification \cite{MarecekTakac2015}, until we can guarantee local convergence of the Newton method on the non-convex Lagrangian of the problem, 
possibly considering some regularisation \cite{MarecekTakac2015}. 
In particular, the guarantee considers points $z_0$ and $z^*$, such that when we start a Newton method or a similar algorithm 
 at the point $z_0$, %
it will generate a sequence of points $z_i$ converging to $z^*$ with quadratic rate of convergence, i.e.
%\todo[inline]{maybe say local optimum, stationary point???}
%\reply[inline]{Jie: Yes, Newton actually only guarantees fast local convergence; however we would like to consider a backtrack strategy for global optimum in a big framework although in experiments we did not do that. I guess we may illustrate this if there is enough space ...}
\begin{align}
\|z_i - z^* \| \le (1/2)^{2^i - i} \|z_0 - z^*\|.
\end{align}
The guarantee requires only the knowledge of the particular Lagrangian and its partial derivatives at $z_0$, but does not require the computation of  $z_i, i > 0$ or solving of any additional optimisation problems.
This could be seen as means of on-the-fly choice of the solver, which preserves the global convergence guarantees associated 
with convexification, whilst improving upon the convergence rate of first-order methods.

\section{A Brief Overview of $\alpha$-$\beta$ Theory}

Our approach is based on the $\alpha$-$\beta$ theory 
of Smale \cite{ShubSmale1993,chen1994approximate,Cucker1999},
which is also known as 
the point estimation theory.
We present the basics of the theory and a simple illustration, prior to presenting the main result.

Consider the case of a general real-valued polynomial system $f:\R^m \mapsto \R^n$, i.e.,
 a system of polynomial equations $f \coloneq ( f_1, \, \dots, \, f_n)$ in variables $x \coloneq (x_1, \dots, x_m) \in \R^m$.
Let us define the Newton operator at $x \in \R^m$ as 
$$N_f(x) \coloneq x - [\nabla f(x)]^\dagger f(x),$$
where $[\nabla f(x)]^\dagger \in \R^{m \times n}$ is the Moore–Penrose inverse of the Jacobian matrix of $f$ at $x$.
A sequence with initial point $x_0$ and iterates of the Newton method subsequently, $x_{i + 1} \coloneq N_f(x_i)$ for $i \ge 0$,
is well-defined if $[\nabla f(x_i)]^\dagger$ is well defined at all $x_i, i \ge 0$.
We say that $x \in \R^m$ is an \emph{approximate zero} of $f$ if and only if
\begin{enumerate}
    \item the sequence $\{x_i\}$ is well-defined; and
    \item there exists $x' \in \R^m$ such that $f(x') = 0$ and $\| x_i - x' \| \leq (1/2)^{2^i - 1} \|x_0 - x'\|$ for all $i \geq 0$.
\end{enumerate}
We call $x' \in \R^m$ the \emph{associated zero} of $x \in \R^m$ and 
say that $x$ \emph{represents} $x'$.
The key result of $\alpha$-$\beta$ theory is:

\begin{proposition}[\cite{ShubSmale1993,Cucker1999}]
Let $f:\R^m \mapsto \R^n$ be a system of polynomial equations and define functions $\alpha(f,x), \beta(f,x), \gamma(f,x)$ as:
\begin{subequations}
\begin{align}
    \alpha(f,x) &\coloneq \beta(f,x) \gamma(f,x), \label{alpha} \\
    \beta(f,x) &\coloneq \left\| [\nabla f(x)]^\dagger  f(x)  \right\| = \| x - N_f(x) \|, \label{beta} \\
    \gamma(f,x) &\coloneq \sup_{k > 1} \left\| {\frac{{ [\nabla f(x)]^\dagger [\nabla^{(k)}f] (x)}}{{k!}}} \right\|^{1/(k-1)}, \label{gamma}
\end{align}
\end{subequations}
where $[\nabla f(x)]^\dagger \in \R^{m \times n}$ is the Moore–Penrose inverse of the Jacobian matrix of $f$ at $x$ and $[\nabla^{(k)}f] $ is the symmetric tensor whose entries are the $k$-th partial derivatives of $f$ at $x$.
Then there is a universal constant $\alpha_0 \in \R$ such that if $\alpha(f, x) \leq \alpha_0$, then $x$ is an approximate zero of $f$.
Moreover, if $x'$ denotes its associated zero, then $\| x - x'\| \le 2 \beta(f,x)$.
It can be shown that $\alpha_0 = \frac{13 - 3 \sqrt{17}}{4} \approx 0.157671$ satisfies this property.
\label{alphabeta}
\end{proposition}

We refer to \cite{ShubSmale1993,Cucker1999} for the proof and a variety of extensions.
Considering that \cite{ShubSmale1993} is somewhat difficult to read and a part of 
a five-paper series, we refer to the survey of Cucker and Smale \cite{Cucker1999} 
or the very recent survey of Beltran and Pardo \cite{BeltranPardo2009} for an
overview.

Let us illustrate the approach on alternating-current power flows (ACPF),
where the instance is defined by:
\begin{itemize}
\item the graph $G = (V, E)$,
where $V, |V| = n$ is partitioned into $\G$, $\Lo$, and $\{ \Sb \}$ slack buses, and 
adjacent buses $(i, j) \in E$ are denoted $i \sim j$, and 
\item
the admittance matrix  $Y \in \Com^{n\times n}$, with $G :=\textrm{Re}(Y)$, $B :=\textrm{Im}(Y)$
\item 
active and reactive injection $P_i$ and $Q_i$ at the bus $i \in \Lo \cup \{ \Sb \}$. 
\end{itemize}
Following \cite{Dj15}, we define the power-flow operator $F:\R^{2n} \mapsto \R^{2n}$ in terms of complex voltages $V_i=\Vx_i+\i \Vy_i, i\in V$ with $\Vc$, with $\Vc$ stacked as $\Vc_i=\Vx_i,\Vc_{n+i}=\Vy_i$:
%(\ref{eq:F2a}--\ref{eq:F2c}).
%\begin{figure*}[!h]
% IEEE uses as a separator
%\vspace*{4pt}
%\hrulefill
%\normalsize
% Store the current equation number.
%\setcounter{MYtempeqncnt}{\value{equation}}
%\setcounter{equation}{7}
\begin{subequations}
\begin{align}
[F\br{\Vc}]_i & \coloneq G_{ii}\br{\powb{\Vx_i}{2}+\powb{\Vy_i}{2}}  \label{eq:F2a} \\ &
 -\sum_{j \sim i} B_{ij}\br{\Vy_i\Vx_j-\Vx_i\Vy_j}  \nonumber \\ &
 -\sum_{j\sim i} G_{ij}\br{\Vx_i\Vx_j+\Vy_i\Vy_j}-P_i ,i\in V \nonumber\\
[F\br{\Vc}]_{n+i} & \coloneq B_{ii}\br{\powb{\Vx_i}{2}+\powb{\Vy_i}{2}} \label{eq:F2b} \\ &
+\sum_{j\sim i}B_{ij}\br{\Vx_i\Vx_j+\Vy_i\Vy_j} \nonumber\\&
+\sum_{j\sim i}G_{ij}\br{\Vy_i\Vx_j-\Vx_i\Vy_j}-Q_i ,i\in\Lo \nonumber \\
[F\br{\Vc}]_{n+i} & \coloneq \powb{\Vx_i}{2}+\powb{\Vy_i}{2}-\Vset_i^2,i\in\G\label{eq:F2c}
\end{align}\label{eq:F}	
\end{subequations}
% Restore the current equation number.
%\setcounter{equation}{\value{MYtempeqncnt}}
%\end{figure*}
%We will use $F\br{\Vc}=0$ to denote the power flow operator.
%the power-flow operator $F$ (\ref{eq:F2a}--\ref{eq:F2c}) in terms of voltages $V$, i.e. the power-flow equalities within ACOPF, as in \cite{corollaryj15}. 

Whether a point is in \emph{a} domain of monotonicity can be tested by the simple comparison of $\alpha$ and $\alpha_0$: % without solving (\ref{EqDjA}--\ref{EqDjF}):

% TODO: Jakub to polish
\begin{proposition}
% MARTIN, JIE: ACPF is the feasibility part, this is not a typo.
For every instance of ACPF, there exists 
a universal constant $a_0 \in \R$ and a function $\alpha$ of the instance of ACPF and a vector $x \in \R^m$ such that if $\alpha(F, x) \le \alpha_0$, then $x$ is an approximate zero
of $F$. 
\end{proposition}

\begin{proof}
One can either Proposition \ref{alphabeta} to a problem in $\Vc$,
which stacks the real and imaginary parts of the complex-valued vector to obtain a real-valued
problem,
or one may apply an extension of the proposition to complex-valued polynomials, such as Theorem 4.3 in \cite{DEREN1995}.
\end{proof}

Obviously, one needs to compute $\beta$ \eqref{beta} and $\gamma$ \eqref{gamma} to compute $\alpha$ \eqref{alpha}. 
Because $\gamma(f,x)$ is difficult to compute in practice, we wish to establish a bound, e.g., when $m = n$.
Let us first define some auxiliary quantities, which will be used in the following proposition. % ~\cite{alphacertified}
Define a pseudo-norm $\| \cdot \|_1$ on $\R^n$ by 
$\| x \|_1^2 \coloneq 1 + \sum_{i=1}^n |x_i|^2$,
along with the auxiliary diagonal matrix $\Delta_{(d)}$ with entries $\Delta_{(d)}(x)_{i,i} \coloneq d_i^{1/2}\|x\|_1^{d_i-1}$.
Let us consider the degree-$d$ polynomial $g(x) \coloneq \sum_{|\nu|_p \leq d} g_{\nu} x^{\nu}$
where the coefficients $g_{\nu} \in \R$ and $x^{\nu} \coloneq x_1^{\nu_1} \cdots x_n^{\nu_n}$ with $|\nu|_p \coloneq \sum_{i=1}^n \nu_i$.
We can define the following norm:
\[
    \|g\|_p^2 \coloneq \sum_{|\nu|_p \leq d} |g_\nu|^2 \frac{\nu!(d-|\nu|)!}{d!},
\]
%\todo[inline]{what is $a_\nu$? should it be $g_\nu$? Jie, HW: show that this is a norm!}
%\reply[inline]{Jie: Corrected. Reference is added. Proof is simple since $|g_\nu|$ satisfy absolute scalability and triangular inequality and zero vector property.}
where $\nu! \coloneq \prod_{i=1}^n \nu_i!$.
Next, we define a norm on the polynomial system $f$ by simply writing
\[
    \|f\|_p^2 \coloneq \sum_{i=1}^n \|f_i\|_p^2.
\]
Finally, define
\[
    \mu(f,x) \coloneq \max \{1, \|f\|_p \cdot \|[\nabla f(x)]^\dagger  \Delta_{(d)}(x)\|\}.
\]
With these quantities, we arrive at the following proposition bounding $\gamma(f,x)$:

\begin{proposition}[\cite{ShubSmale1993,Cucker1999}]  
    \label{mubound} %alphacertified
Let $f:\R^{n} \mapsto \R^{n}$ be a polynomial system with $d_i \coloneq \degree(f_i)$ and $D \coloneq \max_i \{d_i\}$.
If $x\in\R^n$ such that $[\nabla f(x)]$ is invertible, then
\begin{equation}
\gamma (f, x) \leq \frac{\mu(f, x) D^{3/2}}{2\|x\|_1}.
\end{equation}
%\todo[inline]{$\mu$ was defined first for $f$ and then for $x$, ie $\mu(f,x)$. Here it is switched....}
%\reply[inline]{Jie: to be consistent, I corrected it to (f, x).}
\end{proposition}

Notice that %the right-hand side is easily computable, but that 
the proposition assumes a 
polynomial system, rather than a polynomial optimisation problem.

\section{The Theory}

We extend the approach to polynomial optimisation problems (POP).
%in two ways, which differ in the treatment of inequalities.
Considering the developed insights \cite{lasserre2013lagrangian} into the availability and strength of 
certain Lagrangian relaxations of a POP, 
we derive a test where knowing only the relaxation and its derivatives at a particular point, 
we can decide whether one can switch to the Newton method on the polynomial relaxation. 
Although there are many options for picking the relaxation, we suggest  
to track the active set and wait until it stabilises. 
Then, one may consider a polynomial,
in whose construction inequalities in the active set are treated as equalities, 
while the remaining inequalities are disregarded. 
%This polynomial and its derivatives can again be used to decide whether one can switch to the Newton method.
Notice that unless one runs the Newton method on that very polynomial, one may need to back-track, 
  whenever the active set changes while running the Newton method.
%This part is useful, inasmuch it can be implemented efficiently,
%albeit less elegant. 

\subsection{The Preliminaries}

In order to describe the approaches formally, we introduce some notation.
Let us denote the polynomial ring over the reals by $\R[x]$ and consider the compact basic semi-algebraic set $\K$ defined by:
\begin{align}
\label{setk}
\K\, \coloneq \,\{\,x\in\R^m \::\: & g_j(x)\,\geq\,0,\quad j=1,\ldots,p, \\ \notag
                          & h_k(x)\,=\,0,\quad k = 1, \ldots, q \} 
%                          & h_l(x)\,>\,0\quad  l = 1, \ldots, r
\end{align}
for some  $g_j\in\R[x]$, $j=1,\ldots,p$ in $x \in \R^m$,
          $h_k\in\R[x]$, $k=1,\ldots,q$.
The corresponding polynomial optimization problem (\emph{POP}) is:
\begin{equation}
P:\quad f^* \coloneq \displaystyle\min_{x \in \R^m}\:\{f(x)\::\: x\in\K\:\}
\label{pop}
\end{equation}
where $f\in\R[x]$. %is again a polynomial in $\x \in \R^m$. 
We use $f^*$ to denote the value of the objective function $f$ at the optimum of the POP \eqref{pop};
notice that there need not be a unique point at which $f^*$ is attained.
We use $\mathbb P^m$ to denote the space of all possible descriptions of a POP \eqref{pop}
 in dimension $m$.
For additional background material on polynomial optimisation, we refer to \cite{Handbook}.

In a departure from the tradition, 
we use the term \emph{Lagrangian} loosely, to mean a function ${\tilde L}: \R^{\tilde m} \mapsto \R, \tilde m > m$ associated with 
  a particular instance of a POP \eqref{pop} in $\R^m$.
In the best known example, one has $\tilde m = m + p + q$ and $\tilde x \in \R^{\tilde m}$ is the concatenation of the 
  variable $x\in\R^m$ and the so called Lagrangian coefficients $\lambda$ associated with the constraints:
\begin{align}
L(x, \lambda) \coloneq & f(x) %- \sum_{j=1}^{m} \lambda_j g_j(\x).
+ \sum_{j = 1}^{p} \lambda_j g_j(x) 
+ \sum_{k = 1}^{q} \left( \lambda_k \max\{ 0, h_k(x) \} \right)
\label{Ltrad}
\end{align}
The textbook version \cite{Bertsekas} of a Lagrangian relaxation is:
\begin{align}
\rho_0 \coloneq & \max_{\lambda \in \R^{p+q} } \min_{x \in \R^m} L(x, \lambda)
\label{rho}
\end{align}
and it is known that $\rho_0 \leq f^*$. 
One often adds additional regularisation terms 
\cite{MarecekTakac2015}, which may improve the rate of convergence, but do not remove
the fact that one may have $\rho_0 \ll f^*$.
One may replace the $\max$ in  \eqref{Ltrad} by constraints on $\lambda_k$ to be non-negative in \eqref{rho}, 
but the non-negativity constraints also make it impossible to apply $\alpha$-$\beta$ theory directly. 
%Unless \eqref{rho} happens to very particular (e.g., convex), one can hope to find only stationary points
%$(\bar \x, \bar \lambda) \in \R^N \times \R^m_+$ in polynomial time. 

Using this looser definition of the Lagrangian, we define the domain of monotonicity of a \eqref{pop}, with respect to 
a particular Lagrangian:

\begin{definition}[Monotonicity domain with respect to ${\tilde L}$]
%Let $\x \in \R^N$ and $L: \R^N \to \R$ and 
For any $\tilde x \in \R^{\tilde m}$ and ${\tilde L}: \R^{\tilde m} \mapsto \R$,
consider a sequence ${\tilde x}_0 \coloneq \tilde x$, ${\tilde x}_{i + 1} \coloneq N_{\tilde L}(\tilde x_i)$ for $i > 0$. 
The point $\tilde x$ is within the monotonicity with respect to ${\tilde L}$ 
if this sequence is well
defined and there exists a point $\tilde x' \in \R^{\tilde m}$ such that $\tilde L(\tilde x') = 0$ and
\begin{align}
||{\tilde x}_i - {\tilde x'} || \le (1/2)^{2^i - i} ||{\tilde x}_0 - {\tilde x'}||.
\end{align}
Then, we call $\tilde x'$ the \emph{associated stationary point} of $\tilde x$ 
and say that $\tilde x$ \emph{represents} $\tilde x'$.
\end{definition}

Notice that we use tilde to stress the variable parts, such as the Lagrangian $\tilde L$ and its dimension $\tilde m$.
Notice also that domains of monotonicity are known also as the region of attraction, the basin of attraction, etc.

\subsection{The Assumptions}

Recently, it has been realised that one can approximate the global optimum $f^*$ as closely as possible, 
in case one applies the relaxation to a problem $\tilde{P}$ equivalent to $P$, 
which has sufficiently many redundant constraints.
To state the result, we need some additional technical assumptions:
\begin{assumption}
\label{ass1}
$\K$ is compact and $0\leq g_j(x) \leq 1$ on $x\in \K$ for all
$j=1,\ldots,p$, possibly after re-scaling. %$k=1,\ldots,q$. 
%\todo[inline]{Jakub, I am not sure what is $q$? should we just remove it?}
Moreover, the family of polynomials $\{g_j, 1-g_j\}$ generates the algebra $\R[x]$.
\end{assumption}
Notice that if $\K$ is compact, one may always rescale variables $x_i$ and add redundant constraints $0\leq x_i\leq 1$ 
\cite{josz2016strong} for all $i=1,\ldots,p$, such that 
the family $\{g_j,1-g_j\}$ generates the algebra $\R[x]$ and Assumption \ref{ass1} holds.
Further, we assume:

\begin{assumption}
\label{ass2}
There exists a unique point $x^* \in \K$, where $f^*$ is attained.  
\end{assumption}
Notice that one can easily construct an example with two generators 
and a single other bus, where this assumption is violated. At the same
time, it is easy to see that an arbitrarily small perturbation makes
it possible to satisfy the assumption. 
Alternatively, one could replace Assumption \ref{ass2} with an assumption 
on the separation of stationary points, as discussed in \cite{Corless1997}.

\subsection{The Results}

It is well-known that one can construct:

\begin{lemma}[Lasserre Hierarchy]
\label{convexlagrangian}
Let Assumption \ref{ass1} hold for $K$ \eqref{setk} underlying a POP $P$ with optimum $f^*$.
%with global optimum $f^*$.
%There exists a \emph{convex} $\tilde{P}_d$, $d\in\N$, 
For every $\epsilon>0$, there 
exists $d_\epsilon\in\N$ such that for every $d\geq d_\epsilon$,
there exists a \emph{convex} Lagrangian relaxation of $P$, 
which yields a lower bound $f^*-\epsilon\leq \rho_d \leq f^*$.
\end{lemma}

\begin{proof}
The proof follows from Theorem~3.6 of Lasserre \cite{Lasserre2006}, when 
one considers a Lagrangian relaxation of the semidefinite programs.
There, the strong duality can be assured by a reformulation of the POP, cf. \cite{josz2016strong}.
\end{proof}

Notice, however, that 
these Lagrangians, albeit convex, are not polynomial due to the presence of the semidefinite constraint.
Moreover, for $d \ge 1$, a single iteration of minimising the \emph{convex} Lagrangian, even using a first-order method, can be computationally much more demanding than 
a single iteration of second-order methods for the basic Lagrangian $\rho_0$.
We would hence like to study the domains of monotonicity with respect to the various
other Lagrangians. %optima and stationary points of such Lagrangians.

%\begin{lemma}
%There exists a universal constant $a_0 \in \R$ and a function $\alpha:
%\mathbb P^m \times \R^m \mapsto \R$
%such that 
%for all $m \in \Z, p \in \mathbb P^m,$ and $x \in \R^m$,
%there exists $\tilde{P}_d$, $d\in\N$, such that for every $\epsilon>0$, there 
%exists $d_\epsilon\in\N$ such that for every $d\geq d_\epsilon$,
%the Lagrangian relaxation $\tilde{L}_d$ of $\tilde{P}_d$, 
%yields a lower bound $f^*-\epsilon\leq \rho_d \leq f^*$.
%If $\alpha(\nabla \tilde{L}_d, x) \le \alpha_0$, then $x$ is in the domain of monotonicity of
%stationary point $x'$ of Lagrangian relaxation $\tilde{L}_d$ of $\tilde{P}_d$.
%\label{weakling}
%\end{lemma}

%\begin{proof}
%We apply Proposition \ref{alphabeta} %the reasoning of Chen \cite{chen1994approximate} 
%to first-order conditions of a Lagrangian relaxation of Lemma~\ref{lagrangian} of the 
%polynomial optimisation problem \eqref{pop}.
%\end{proof}

%Extending the reasoning to convex Lagrangians of Lemma~\ref{convexlagrangian}
%is non-trivial, because the convex Lagrangians are not necessary polynomial (analytic),
%as defined. 
%However:

Specifically, notice that one can obtain Lagrangians by distinguishing which inequalities
are satisfied with equality at a particular point. 
At any given point $x$, we can evaluate what inequalities of the 
POP \eqref{pop} are active, i.e., satisfied with equality.
Let us denote the index set active inequalities $A(\cdot) \subseteq I$,
\begin{align}
\label{Aeps}
A(x, \epsilon) & \coloneq \left\{ k \; | \; k = 1, \ldots, q, h_i(x) \leq \epsilon \right\},
%A(x) & \coloneq A(x, 0),
\end{align}
%extending the notation of Section 3.3 of Bertsekas \cite{bertsekas1999nonlinear}.
At an arbitrary point $x \in \R^n$, we can 
evaluate $A(\cdot)$ and 
construct a locally valid, but polynomial
 Lagrangian:
    \begin{align} \label{Lprime}
            L'(x, \lambda) \coloneq & f(x) %- \sum_{j=1}^{m} \lambda_j g_j(\x).
            & + \sum_{k = 1}^{q} \lambda_k h_j(x) 
            & + \sum_{j = 1}^{p} \left( \mathbbm{1}_{j \in A} \lambda_j g_j(x) \right),
    \end{align}

and it is clear that:

\begin{lemma}
\label{activeset}
Let Assumptions \ref{ass1} and \ref{ass2} hold for $\K$ \eqref{setk}.
%For every $P$, There exists $\tilde{P}_d$, such that
For every $\epsilon>0$, there 
exists $d_\epsilon\in\N$ such that for every $d\geq d_\epsilon$,
the Lagrangian relaxation of $\tilde{P}_d$, 
yields a lower bound $f^*-\epsilon\leq \rho_d \leq f^*$ achieved at $x_d^*$ and
the active set $A(x_d^*, \epsilon)$ induces $L'(x, \lambda)$
with optimum $\rho_d$.
\end{lemma}

\begin{proof}
    The proof follows from the reasoning of Propositions 7 and 8 of \cite{Corless1997}, as explained by Henrion and Lasserre \cite{Henrion2005}:
    Under Assumptions \ref{ass1} and \ref{ass2}, the moment matrix for $d$ makes it 
    possible to extract the solution by Cholesky decomposition, which in turn allows
    to estimate the active set.
\end{proof}

This allows for the direct application of $\alpha$-$\beta$ theory:

\begin{theorem}
There exists a universal constant $a_0 \in \R$,
such that for all $m \in \N, P \in \mathbb P^m,$
where Assumptions \ref{ass1} and \ref{ass2} hold for $P$,
there exists a $d\in\N$, such that for every $\epsilon>0$, there 
exists $d_\epsilon\in\N$ such that for every $d\geq d_\epsilon$,
there is a Lagrangian relaxation $\tilde{L}_d$ in dimension $\tilde m$, 
and a function $\alpha:
\mathbb P^{\tilde m} \times \R^{\tilde m} \mapsto \R$
such that 
if $\alpha(\nabla \tilde{L}_d, \tilde x) \le \alpha_0$, then $x$ is the domain of monotonicity of
a solution with objective function value $\rho_d$ such that $f^*-\epsilon\leq \rho_d \leq f^*$.
\label{Theorem1}
\end{theorem}

\begin{proof}
The proof follows from the observation that each \emph{convex} Lagrangian of 
Lemma~\ref{convexlagrangian}
is associated with a
non-convex, but polynomial Lagrangian of Lemma \ref{activeset}, and that both Lagrangians of 
will have a function value at their optima bounded from 
below by $f^*-\epsilon$.
Formally, for all $m \in \N, p \in \mathbb P^m,$ and $x \in \R^m$,
there exists $\tilde{P}_d$, $d\in\N$, such that for every $\epsilon>0$, there 
exists $d_\epsilon\in\N$ such that for every $d\geq d_\epsilon$,
both 
the Lagrangian relaxation $\tilde{L}_d$ of $\tilde{P}_d$, 
and the new Lagrangian relaxation of the same problem $L'_d$
yield a lower bound $f^*-\epsilon\leq \rho_d \leq f^*$.
While minimising the convex Lagrangian of the polynomial optimisation problem \eqref{pop},
we can apply Proposition \ref{alphabeta} %the reasoning of Chen \cite{chen1994approximate} 
to the first-order 
conditions of the corresponding Lagrangian $L'_d$ of Lemma \ref{activeset}.
\end{proof}

In terms of the alternating-current optimal power flow problem (ACOPF),
the theory can be summarised thus:

\begin{corollary}
There exists 
a universal constant $a_0 \in \R$,
such that
for every instance of ACOPF, 
there exists $\delta \in \R, \delta \ge 0$ 
and a function $\alpha: \R^m \mapsto \R$
%$, \beta, f$ of 
specific to the instance of ACOPF, such that for any 
$\epsilon > \delta$ and 
vector $x \in \R^m$ if $\alpha(x) \le \alpha_0$, then $x$ is 
%an approximate stationary point 
in the domain of monotonicity of an optimum of the instance of ACOPF, which is no more than 
$\epsilon$ away from the value of the global optimum with respect to its objective function.
\end{corollary}

\begin{proof}
By Theorem~\ref{Theorem1}. The $\delta$ accounts for the perturbation. 
\end{proof}

In the hybridisation we propose, one starts by solving a convexification
and construction of the active set
in the outer loop.
Then, one may test of the stability of the active set.
Whenever the active set seems stable and the test %of Chen \cite{chen1994approximate} 
of Proposition \ref{alphabeta} applied to $L'$
allows, 
we we switch to the Newton method on the non-convex Lagrangian $L'$.
Some back-tracking line search may be employed withinin the Newton method, 
until a sufficient decrease in $L$ is observed.
Although this algorithm may seem somewhat crude,
it seems to perform well.
%considering the availability of testable sufficient conditions for global optimality \cite{6655990},
%one can often certify the output as being optimal.

Alternatively, one may employ a variant, whose schematic overview is in Algorithm~\ref{alg:HybridActiveset}. 
There, we consider first-order optimality conditions of $L'$ in the test on Line 
\ref{activset-stable-test},
but switch to the Newton method on the first-order optimality conditions of \eqref{Ltrad}, while memorising the current value.
While minimising \eqref{Ltrad}, we 
check the active set; when it does change, 
we revert to solving the convexification with the memorised value.
Although this algorithm may seem even cruder than the above,
it performs better still, in practice. 

%Notice that for the sake of the brevity, the presentation is schematic.
% (e.g., 
%We refer to \cite{MarecekTakac2015} for the measure of infeasibility $T(x, \lambda)$, the definition of the low-rank coordinate descent step, etc.
%does not provide details of the Lagrangian functions for ACOPF, as available from ), 
%and assumes that both the convex and non-convex Lagrangian are in the same dimension (i.e., $\x \in \R^m, \lambda \in \R^m$),
%  which need not be the case, generally.
%Still, it demonstrates the key ideas. %we combine the low-rank coordinate descent algorithm of \cite{MarecekTakac2015} for the convex Lagrangian with Newton steps for the non-convex Lagrangian. 
%Notice that the Newton step (Line \ref{line:Newton} in Algorithm \ref{alg:Hybrid}) may employ the same regularisation %as the convex Lagrangian
% as in \cite{MarecekTakac2015}.
%Unfortunately, we do not have a procedure to compute $d_\epsilon$ of \ref{activeset} in practice, yet.
%Instead, we hence employ a back-tracking mechanism.  
%See 

 \begin{algorithm}[tb]
 \begin{algorithmic}[1]
 \STATE Initialise $x \in \R^m, \lambda \in \R^m$, e.g., randomly
 \FOR{$k \leftarrow 0,1,2,\dots$} 
   \STATE Update $(x, \lambda)$, e.g., using \cite{MarecekTakac2015}
   \STATE Compute the set $A_k$ of inequalities satisfied up to $\epsilon$-accuracy 
   \STATE Construct the polynomial Lagrangian function $L'$ corresponding to $A_k$    
   \IF{ $k > K$ \textbf{and} $A_k = A_{k - 1} = \ldots = A_{k - K}$ \textbf{and} $\alpha(\nabla L', x) \le \alpha_0$ } \label{activset-stable-test} 
       \STATE $S \leftarrow (x, \lambda)$
       \FOR{$l \leftarrow 0,1,2,\dots$}  
         \STATE Update $(x, \lambda)$ using Newton step on $\nabla L = 0$, i.e., minimising \eqref{Ltrad} \label{line:Newton2}
         \STATE Compute the set $A'_l$ of inequalities satisfied up to $\epsilon$-accuracy 
         \IF{$A'_l \not = Ak$}  
           \STATE $(x, \lambda) \leftarrow S$
           \STATE \textbf{break} \label{rejection}
         \ENDIF
         \IF{ Infesibility of $x$ is below $\epsilon$}  
           \STATE Optionally, test sufficient conditions for global optimality, e.g., \cite{6655990}
           \STATE \textbf{break} \label{termination}
         \ENDIF
       \ENDFOR
     \IF{Infesibility of $x$ is below $\epsilon$}
       \STATE \textbf{break}
     \ENDIF
   \ENDIF
  \ENDFOR
 \end{algorithmic}
 \caption{A schema of the hybrid method} % Parallel
 \label{alg:HybridActiveset}
\end{algorithm} 

\section{The Practice}
\label{sec:practice}

%which switches from a convexfication to the Newton method
%on a non-convex problem, whenever we detect it is safe to, 

%\subsection{The Implementation}

 \begin{figure*}[t]
 \center
\includegraphics[scale=0.2]{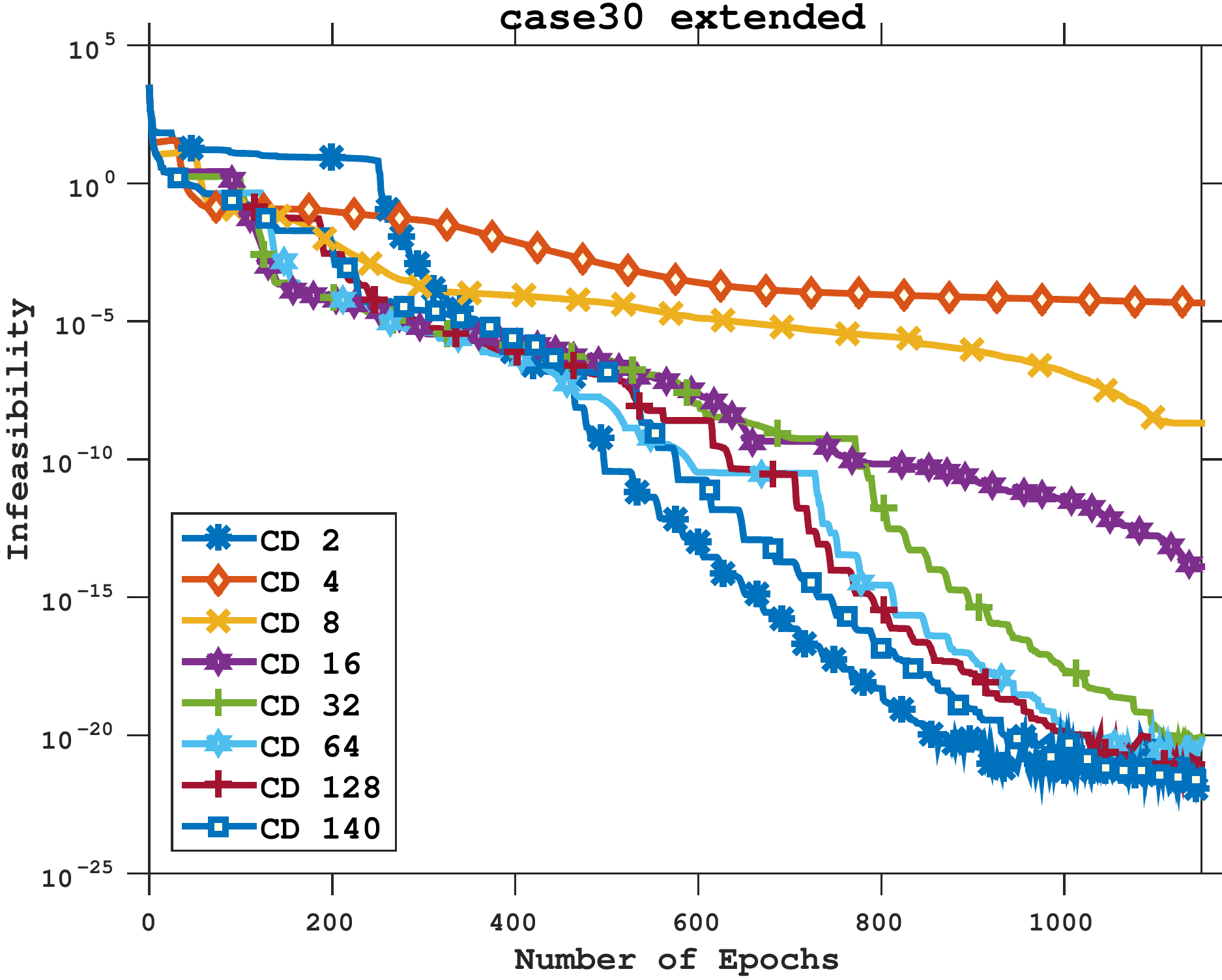}
\includegraphics[scale=0.2]{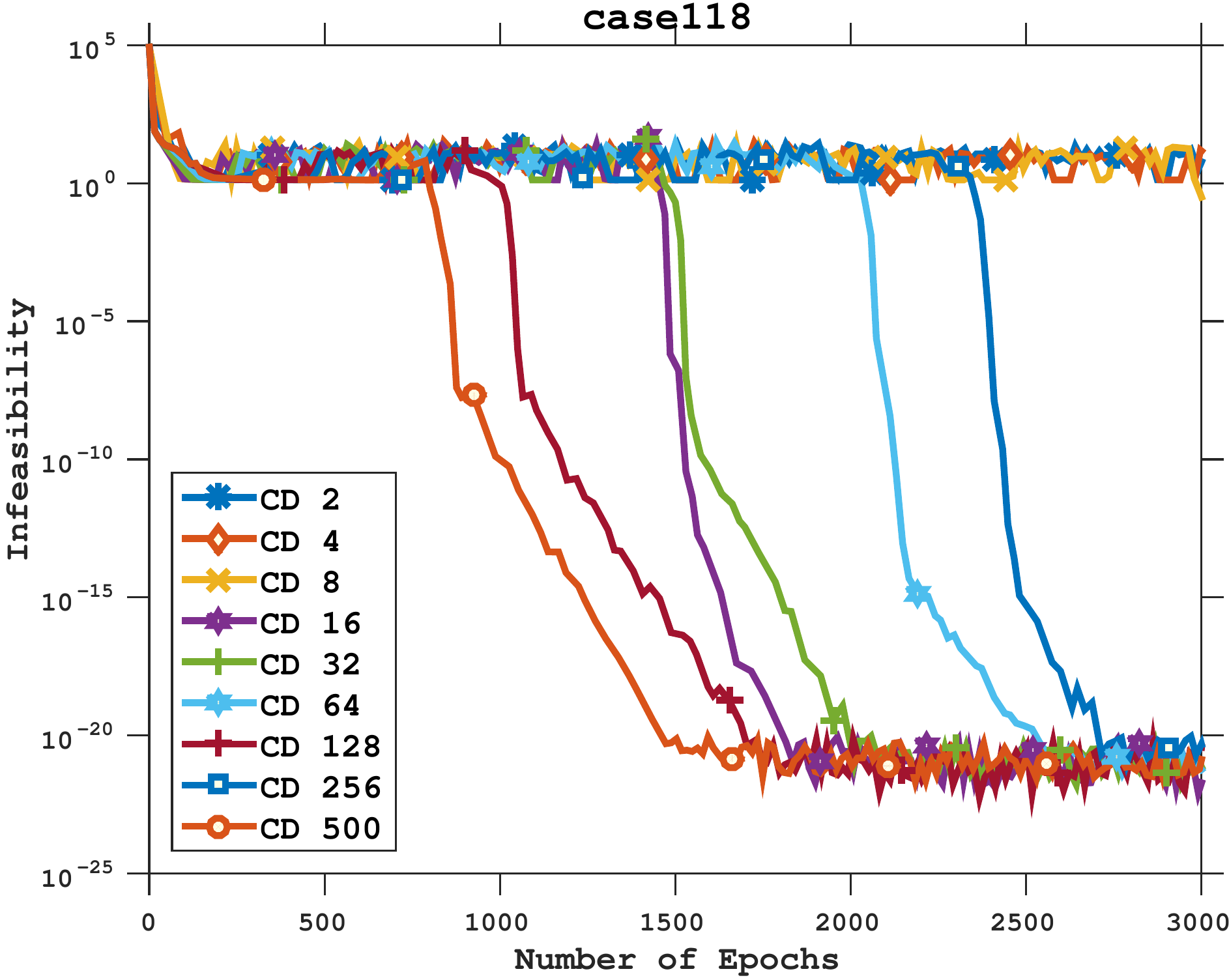}
\includegraphics[scale=0.2]{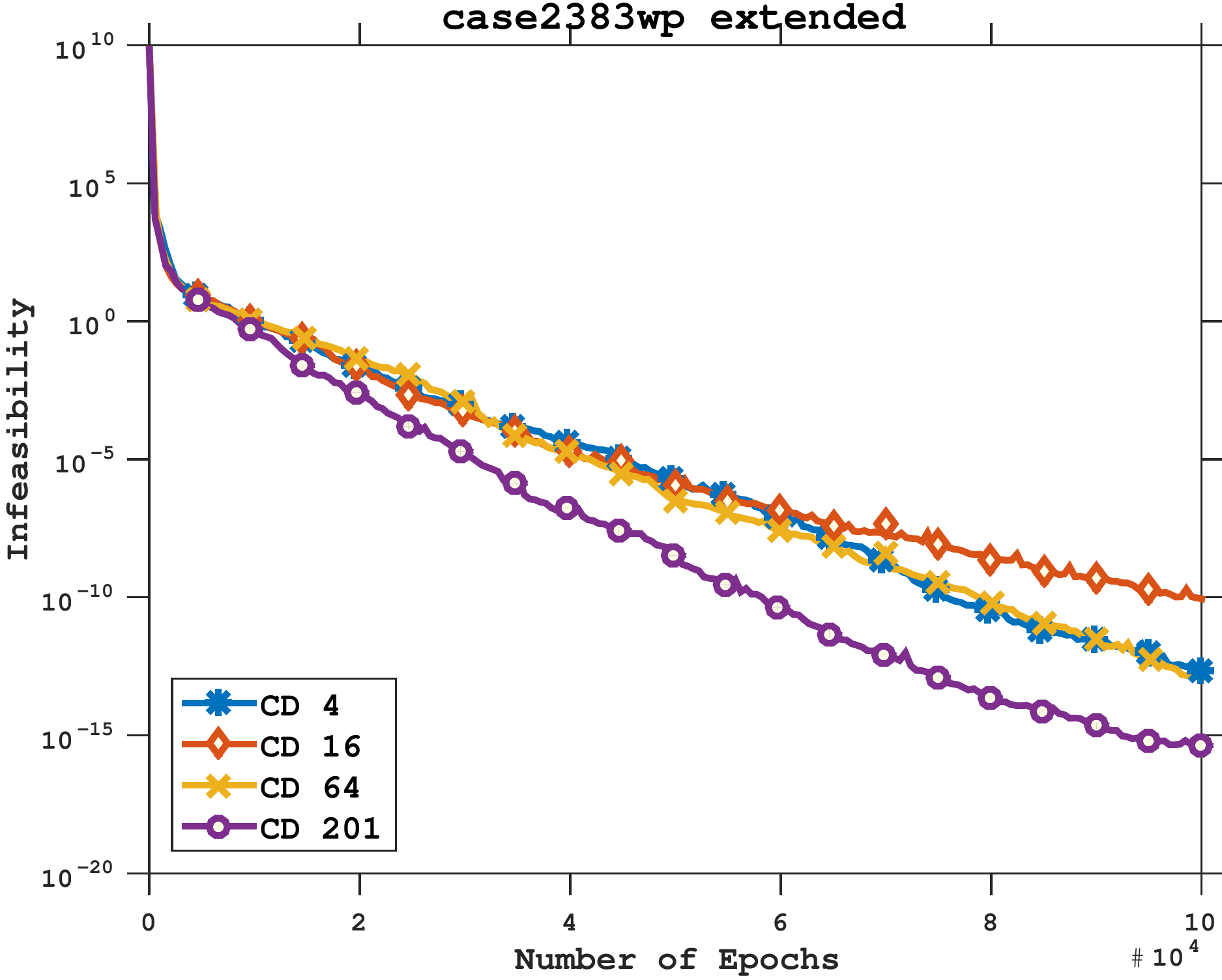}
\\ $ $\\
\includegraphics[scale=0.2]{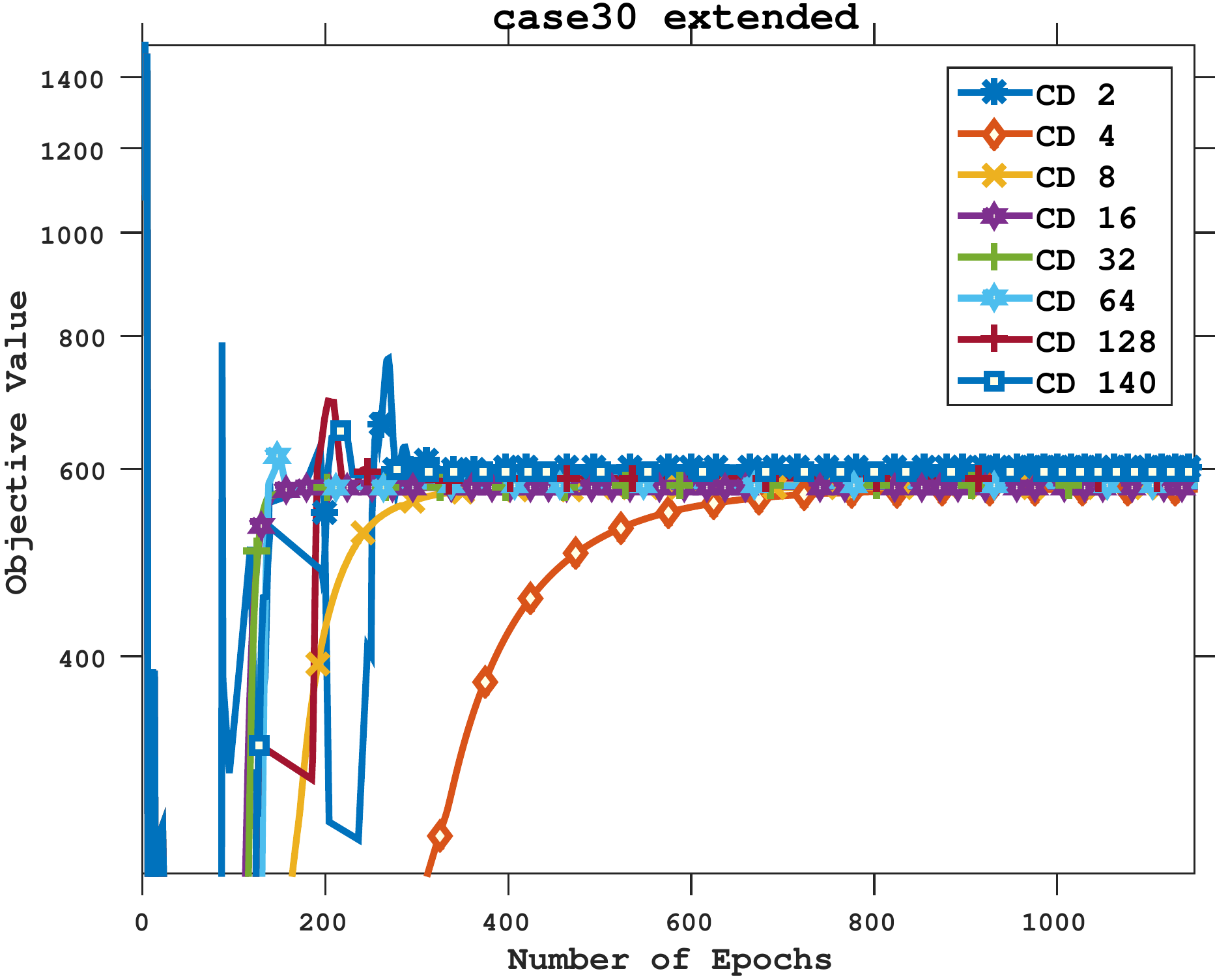}
\includegraphics[scale=0.2]{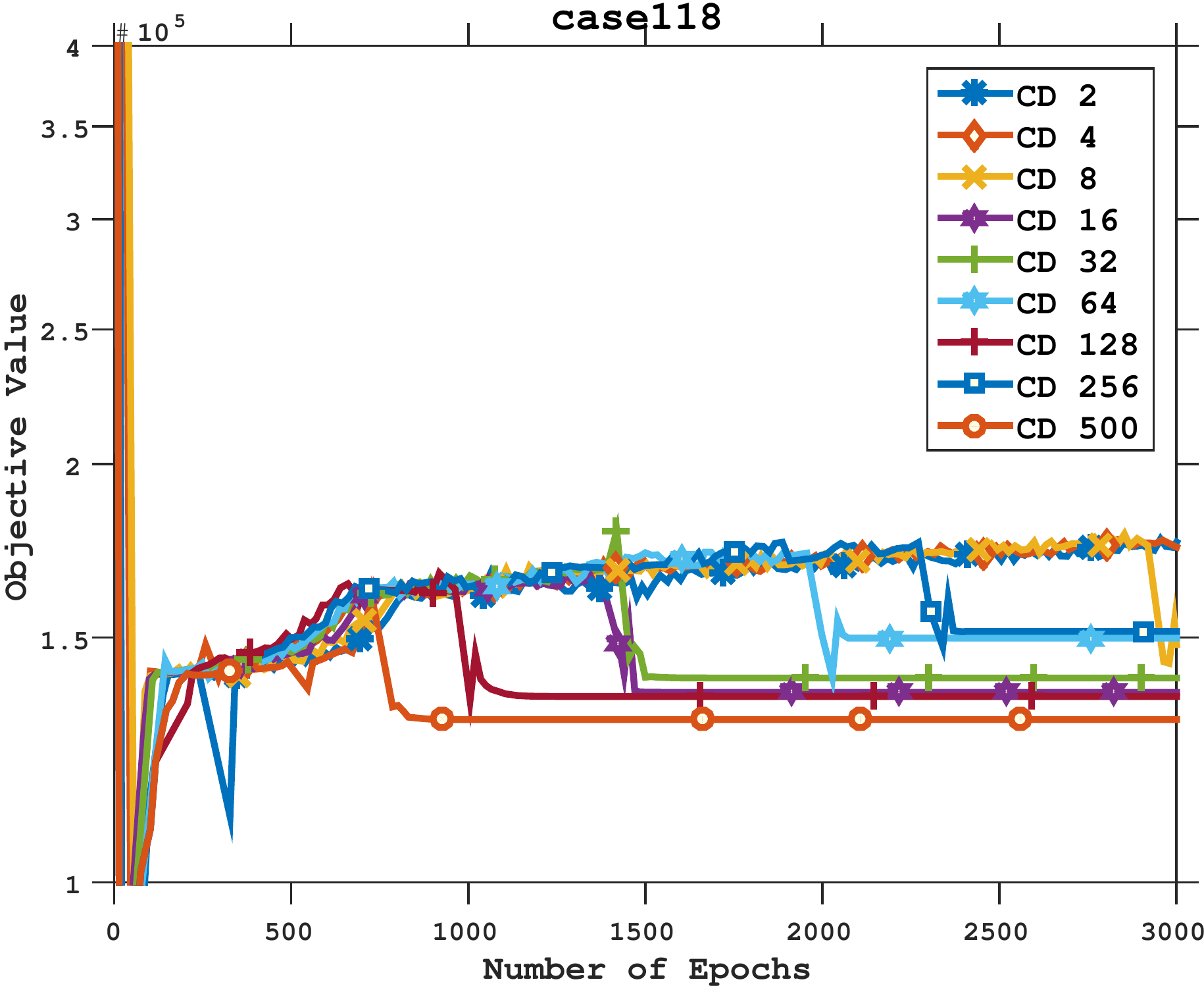}
\includegraphics[scale=0.2]{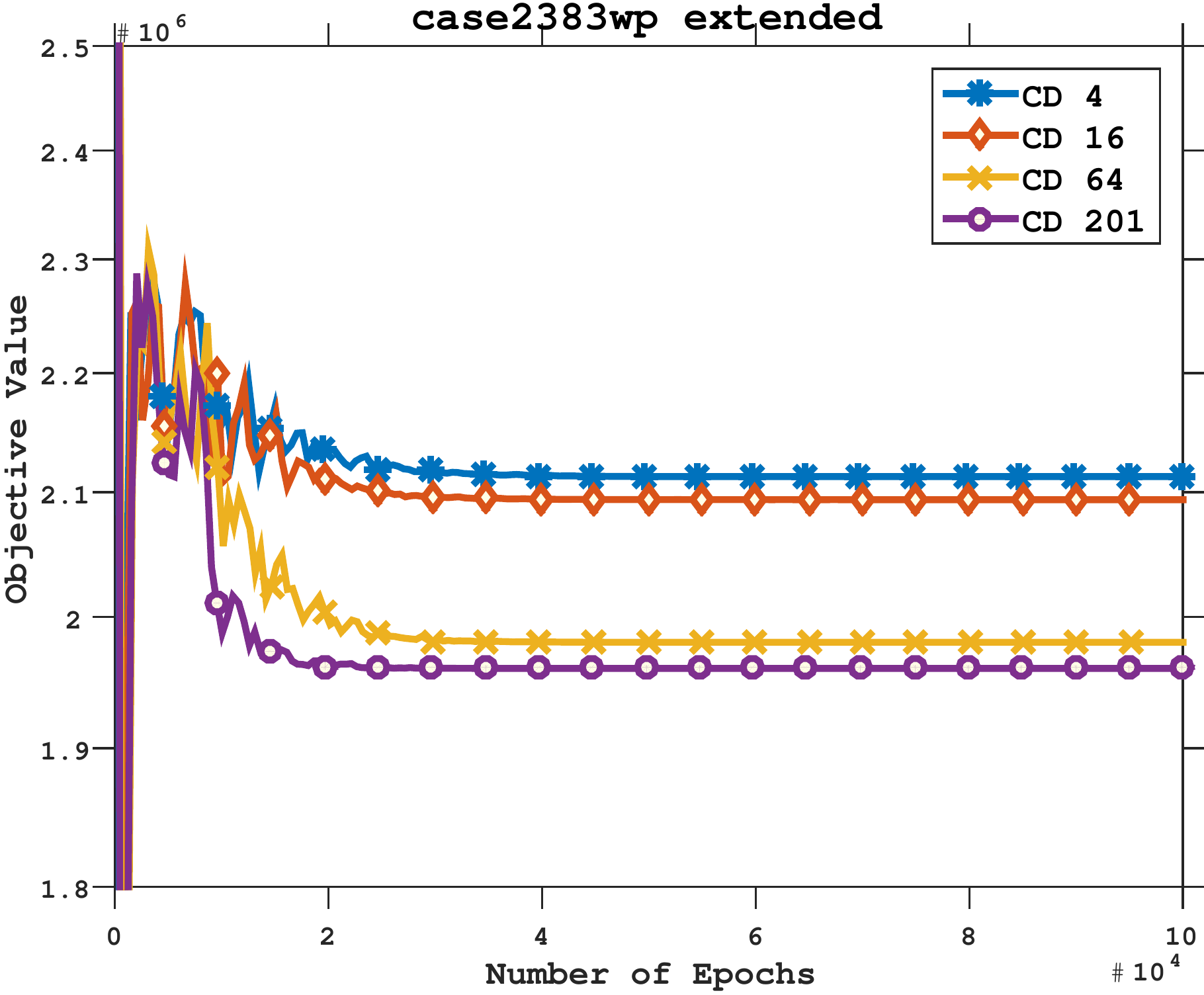}
 \caption{The motivation: the evolution of infeasibility (top row) and objective function (bottom row) when one switches from solving the convexification to the Newton method after a given number of steps on IEEE 30-bus test system (left), 118-bus test system (middle), and a snapshot of the Polish system (case2383wp; right).}
 \label{fig:motivation}
\end{figure*}

In implementing a hybrid method for ACOPF, such as 
Algorithm \ref{alg:HybridActiveset},
one encounters a number of challenges.
One requires a solver for the convexification of ACOPF,
a well-performing implementation of the Newton method for the non-convex Lagrangian $L'$, 
and an implementation of Proposition \ref{mubound}.
We will comment upon these in turn. 
 
The convexification we use is based on of the Lagrangian of the relaxation of Lavaei and Low \cite{lavaei2012zero}.
(As we have shown in \cite{Ghaddar2015},  relaxation of Lavaei and Low is the first level of the hierarchy Lasserre \cite{Lasserre2006}, considered in Lemma \ref{convexlagrangian}.)
In particular, we have used a variant introduced in \cite{MarecekTakac2015}.
To solve it, we have used a problem-specific first-order method \cite{MarecekTakac2015},
which is based on a coordinate descent with a closed-form step. 

The variant of Newton method, which we use, has been implemented specifically for this paper.
Conceptually, it follows the outline of the first-order method \cite{MarecekTakac2015},
but uses the symbolic Hessian in computation of the step.
%For each Newton step, back-tracking line search is used to ensure sufficient decrease. 
%\paragraph{Negative curvature.}
Note that on the non-convex Lagrangian, Newton  direction may turn out not to be a direction of descent.
In dealing with the negative curvature,
 we multiply the direction by -1, %\cite{nesterov1994interior},
 although we could also use a damped variant of Newton method. % \cite{nesterov1994interior}.
%We are experimenting with both.
Multiple Newton steps, each satisfying sufficient decrease, are performed in each iteration of the loop, before a sufficient decrease in the convex Lagrangian is tested. 
 
A key contribution of ours is an implementation of Proposition \ref{mubound} specific to ACOPF.
There, one should observe that $\beta$ is easy to obtain as $\beta(x,L) \coloneq \|d\|_2 = \|L_p\|_2$, where $L_p$ is the Newton direction, and $f = \frac{\partial L}{\partial x}$. By observing  $\frac{\partial L}{\partial x}$, we can use $d_i=3\ \forall i$, thus $D=3$ and 
$\Delta_{(d)}(x) = 3^{1/2}\|x\|_1^2 I_{2n\times 2n}$, where $I_{2n\times 2n}$ is a $2n\times 2n$ identity matrix, so
$$\mu(L, x) = \max \{1, \sqrt{3}\|x\|_1^2\cdot\|\nabla L\|\cdot \|[\nabla^2 L (x)]^{-1}\|\},$$
where the spectral norm $\|[\nabla^2 L (x)]^{-1}\|$ can be computed as the inverse of the smallest eigenvalue of $\nabla^2 L (x)$.
A trivial implementation may run for days even on modest instances. 
In our implementation, we used about 2000 lines of algebraic manipulations in Python to generate considerable amounts of
instance-specific, optimised C code employing Intel MKL Libraries. 
For example, the test for case2383wp involves about 30 MB of C code.
This makes it possible to run the test within seconds even on case2383wp.

%\section{Computational Experiments}

 \begin{figure*}[t]
 \center
\includegraphics[scale=0.2]{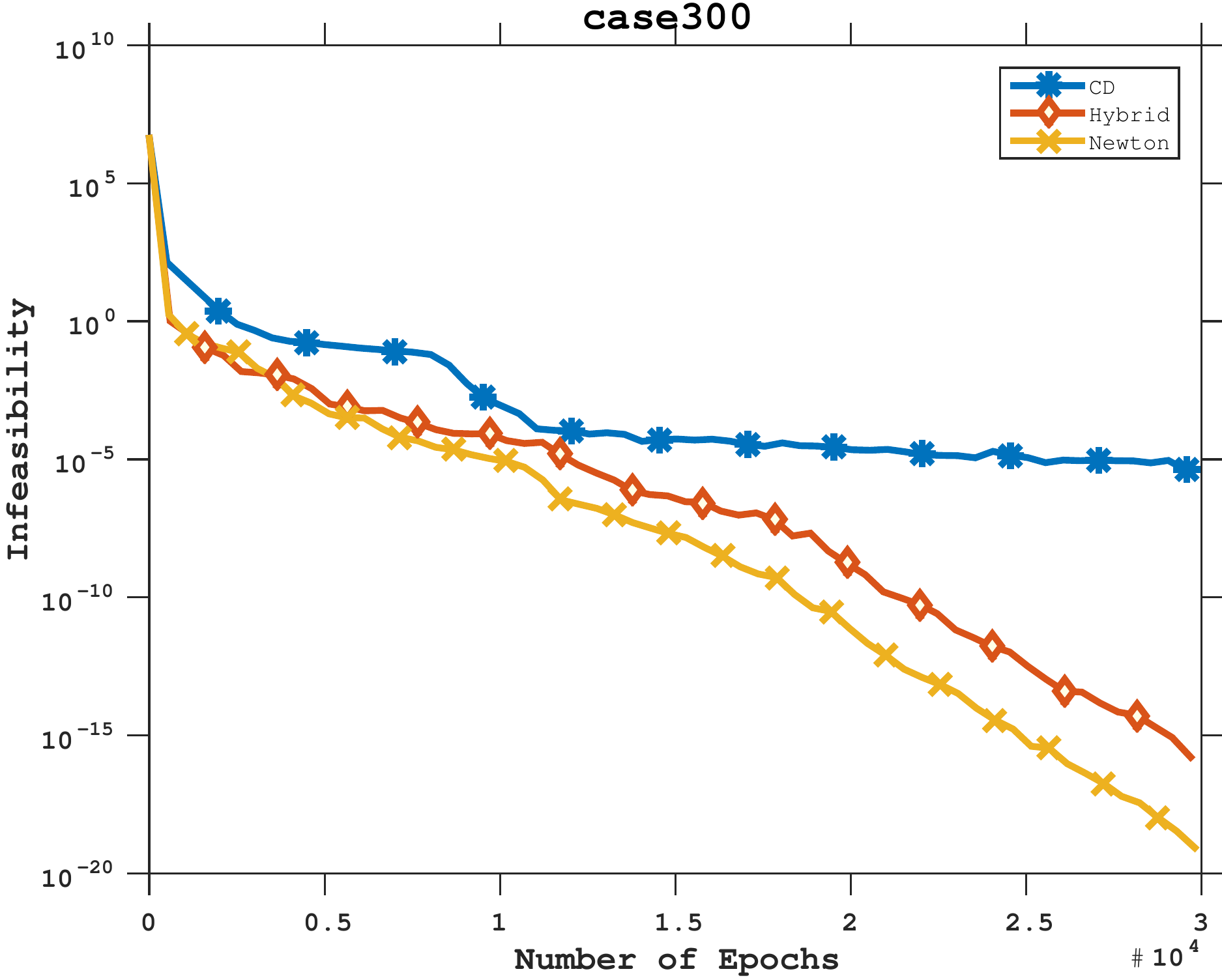}
\includegraphics[scale=0.2]{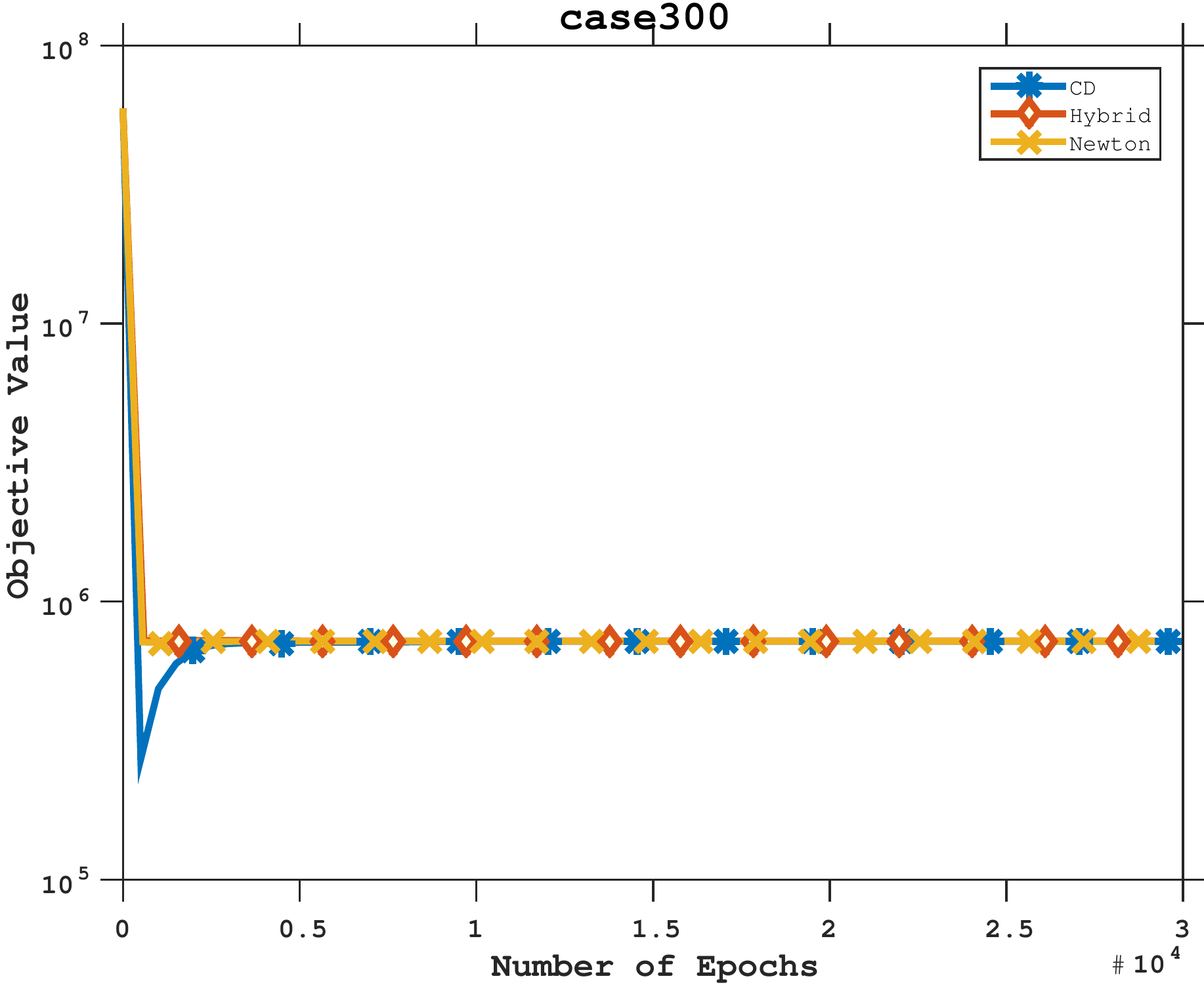}
\includegraphics[scale=0.2]{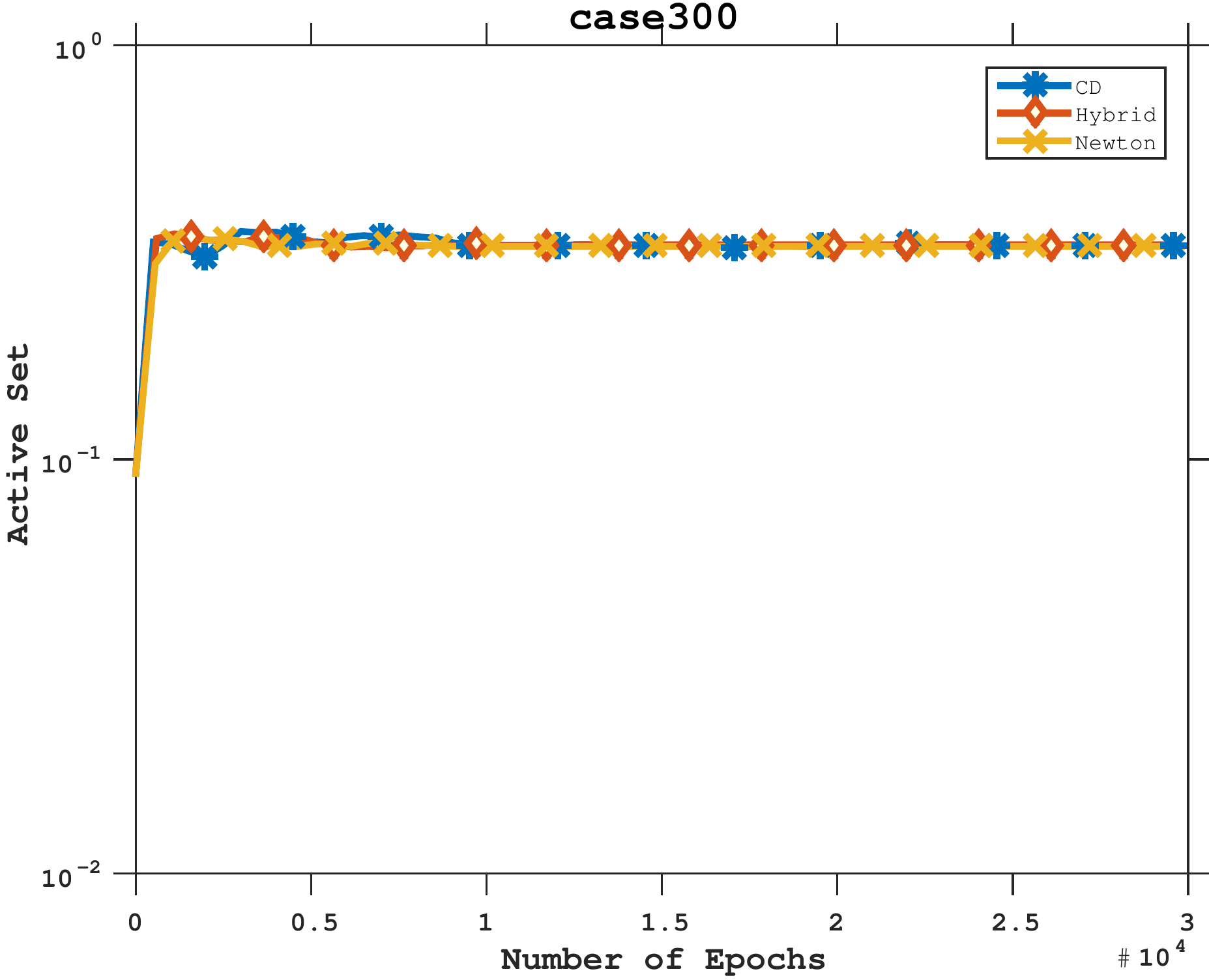}
\\ $ $\\
\includegraphics[scale=0.2]{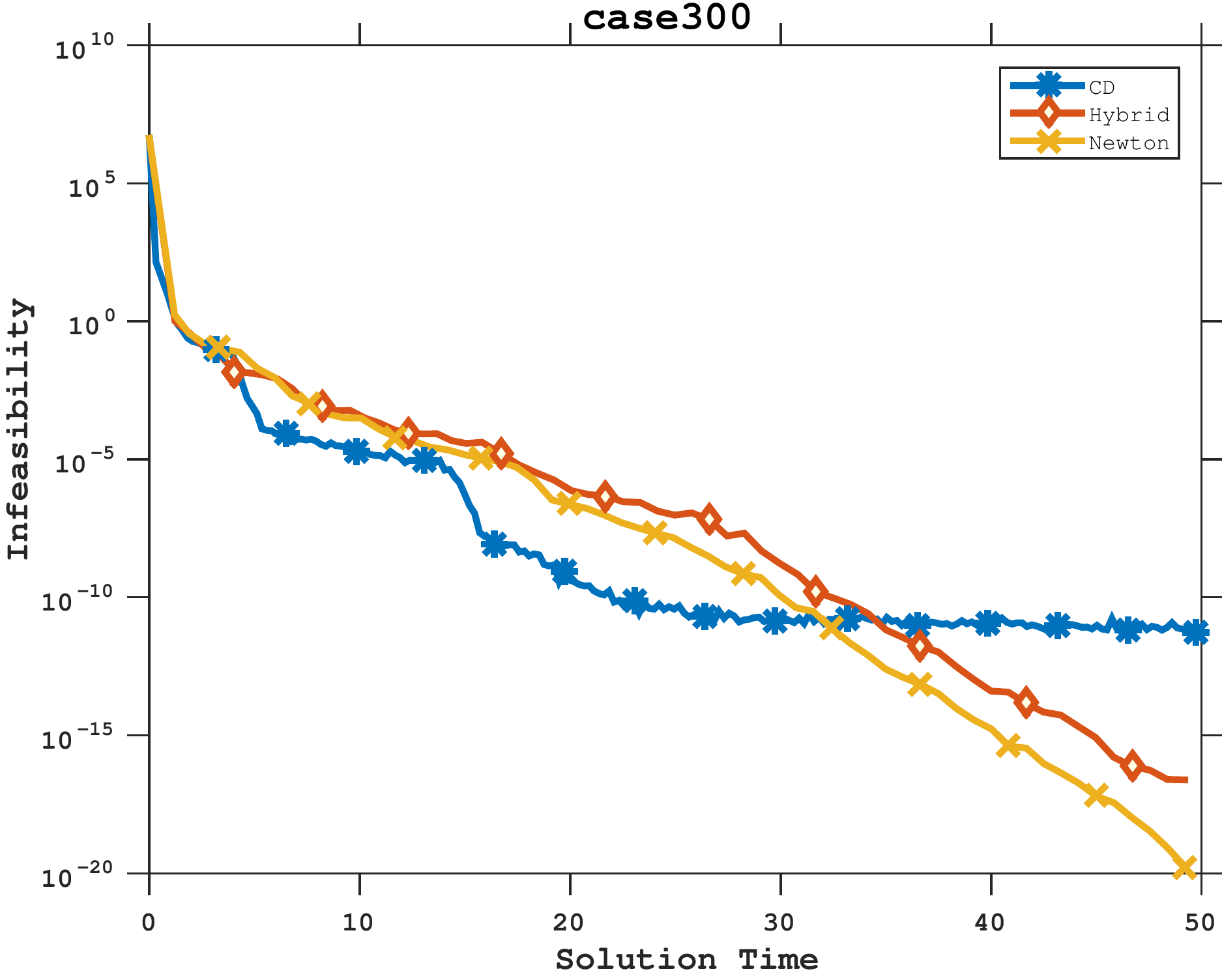}
\includegraphics[scale=0.2]{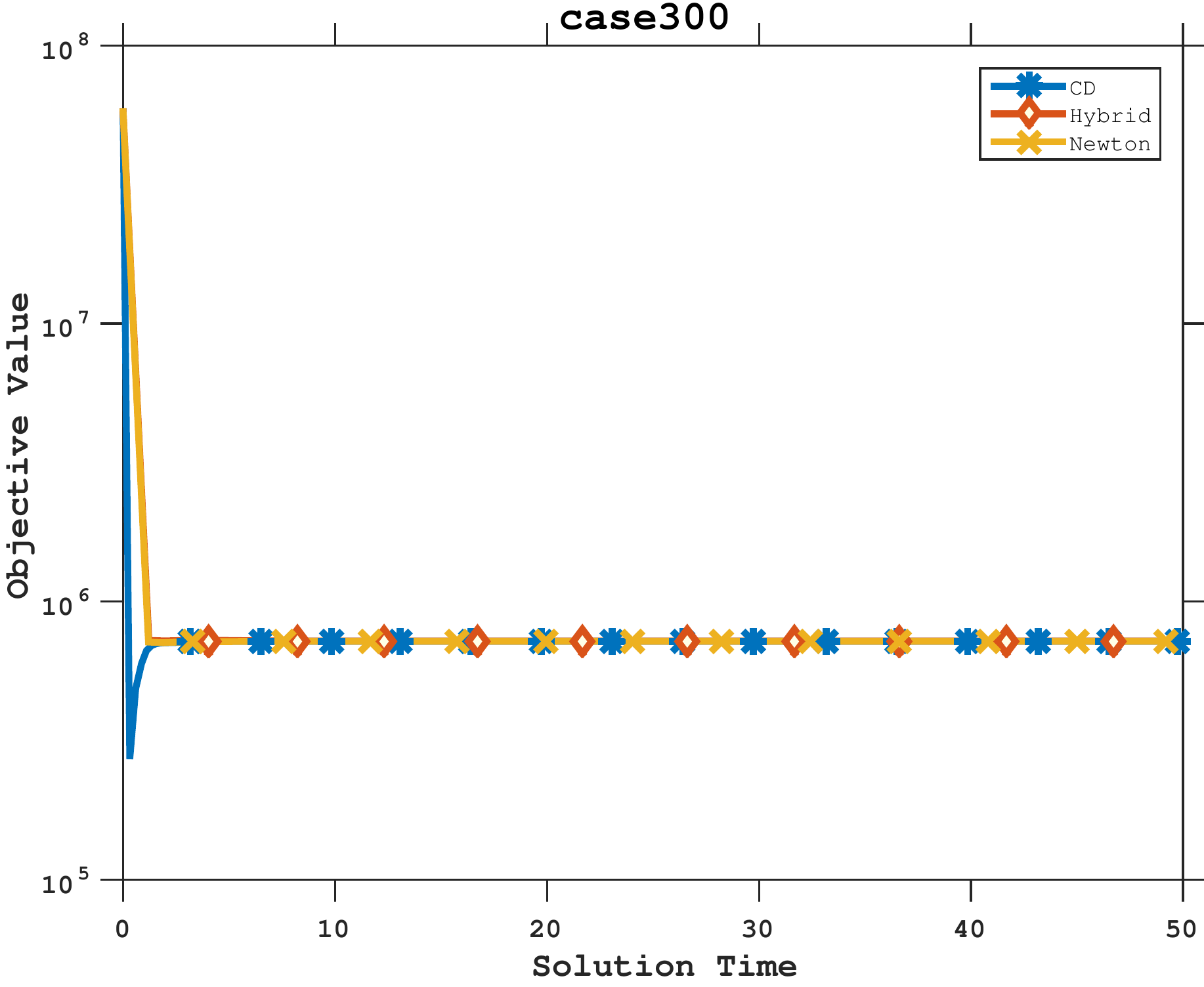}
\includegraphics[scale=0.2]{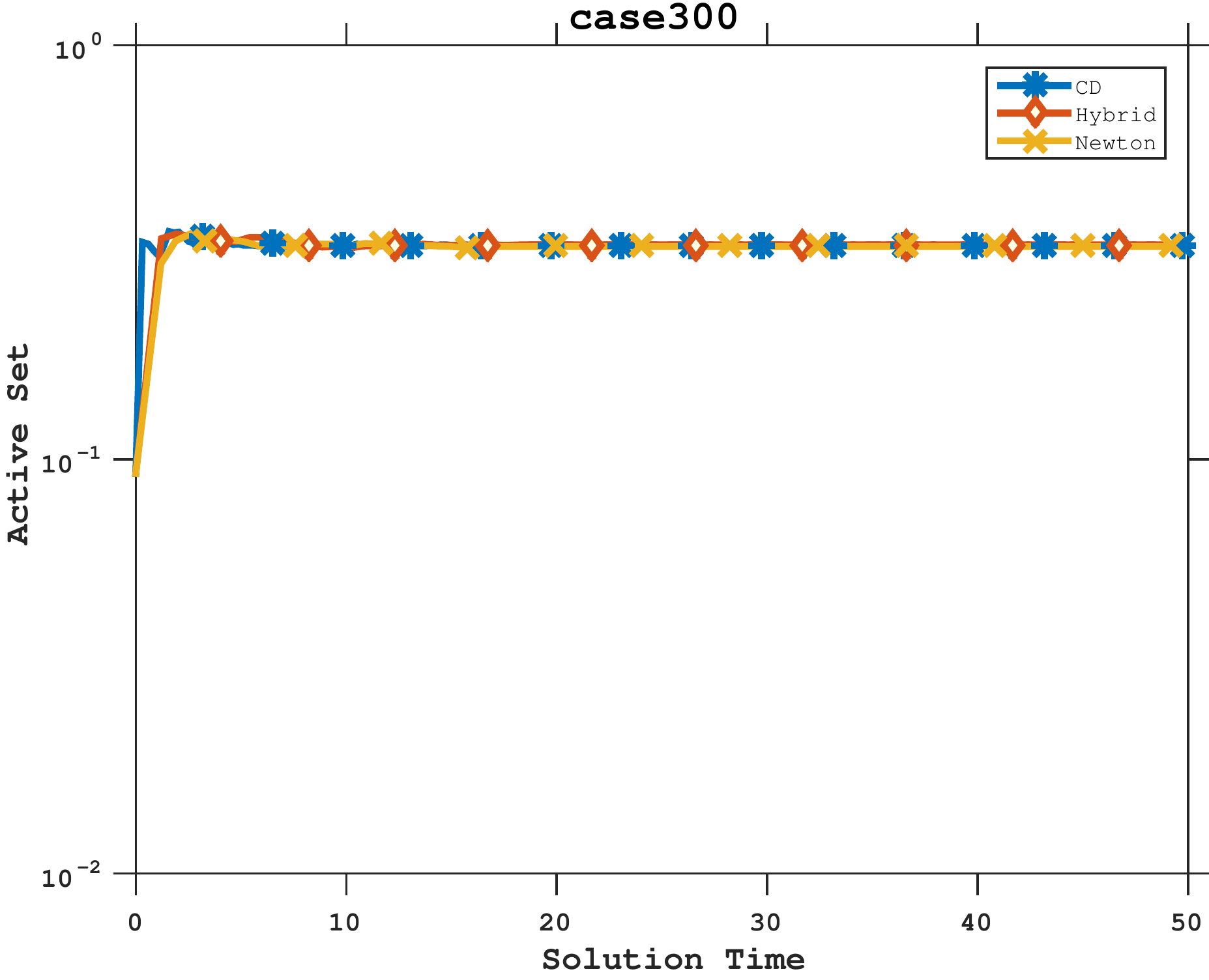}
 \caption{The performance of the hybrid method on the IEEE 300-bus test system (case300).}
 \label{fig:case300}
\end{figure*}

\begin{figure*}[t]
\includegraphics[scale=0.2]{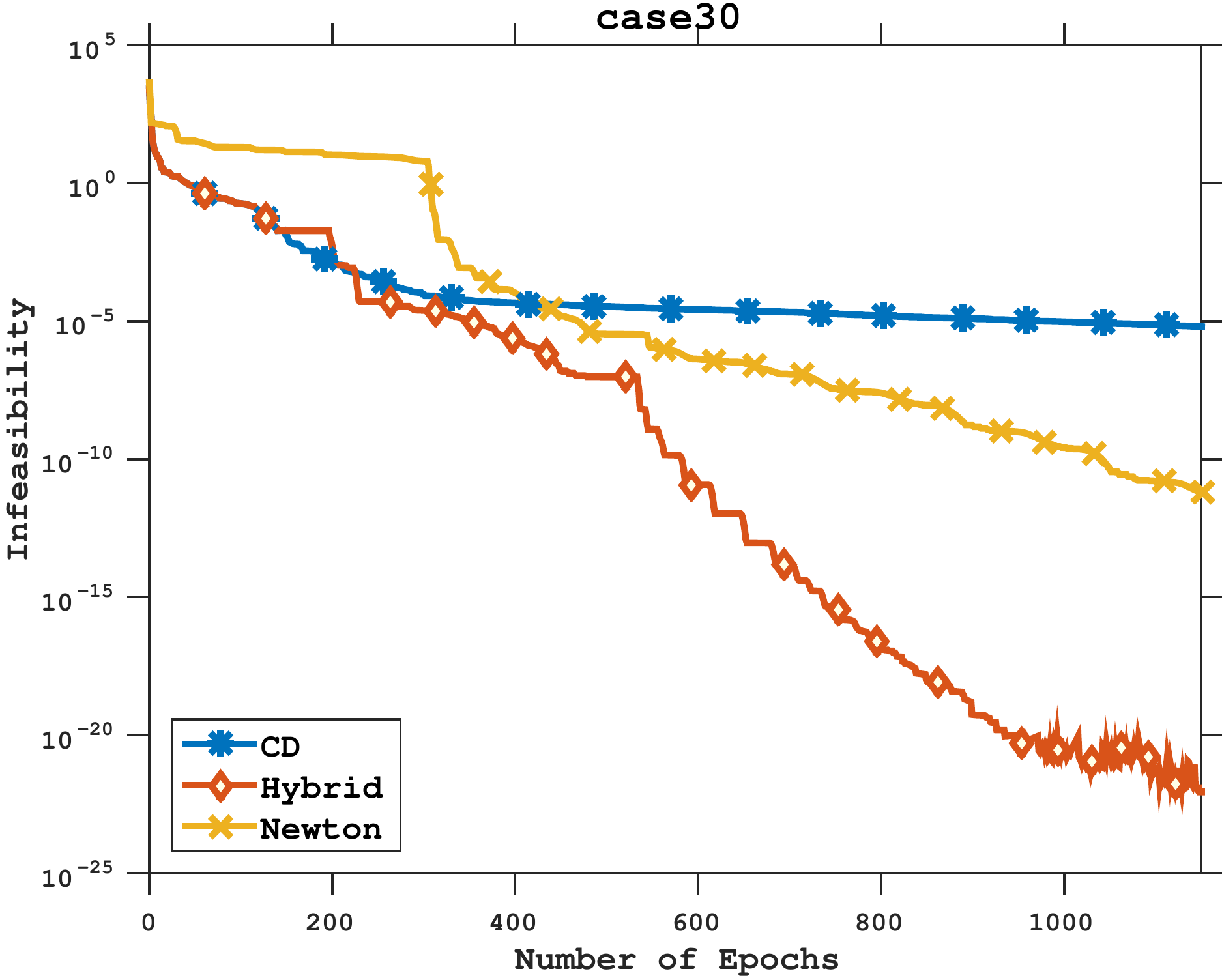}
\includegraphics[scale=0.2]{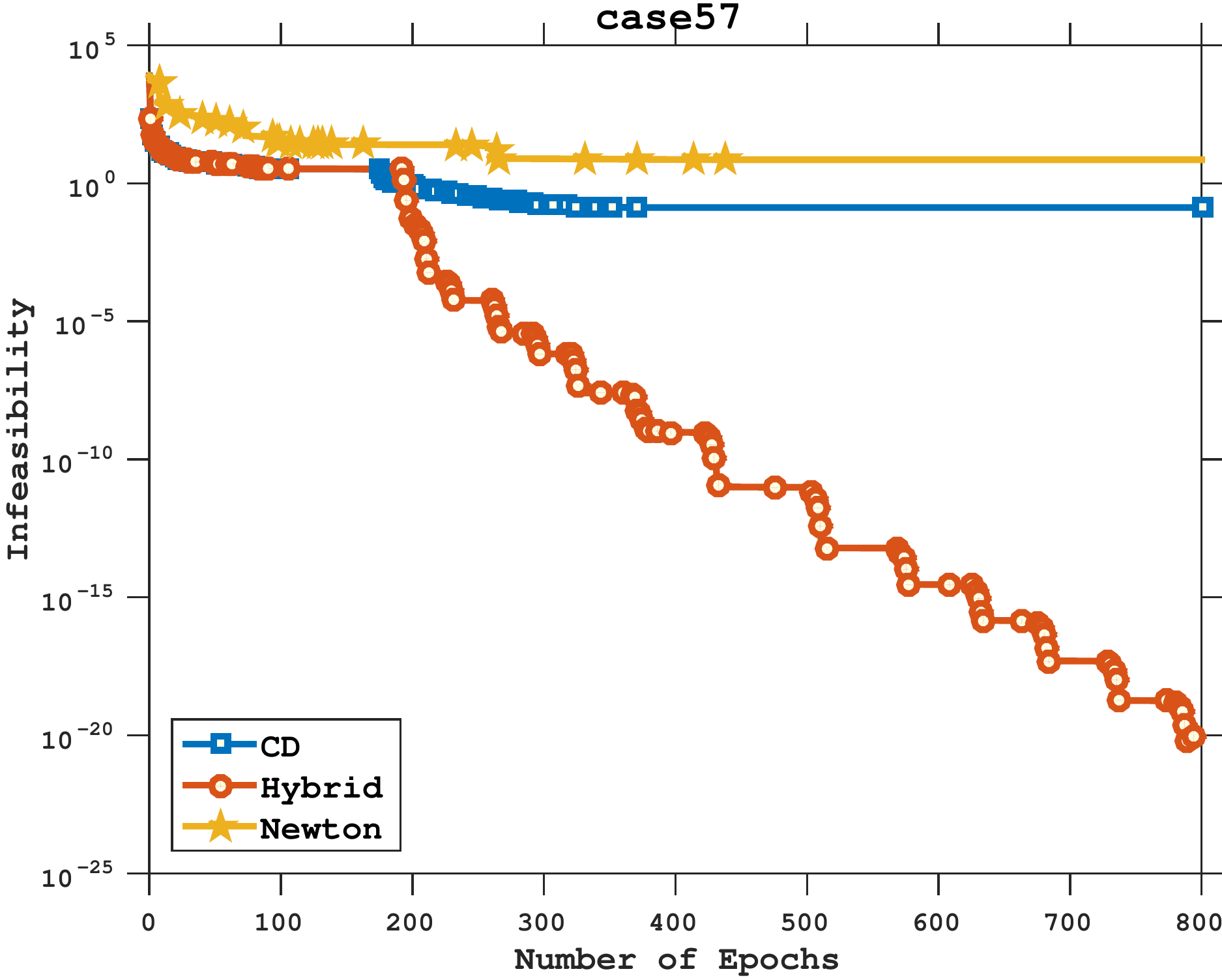}
\includegraphics[scale=0.2]{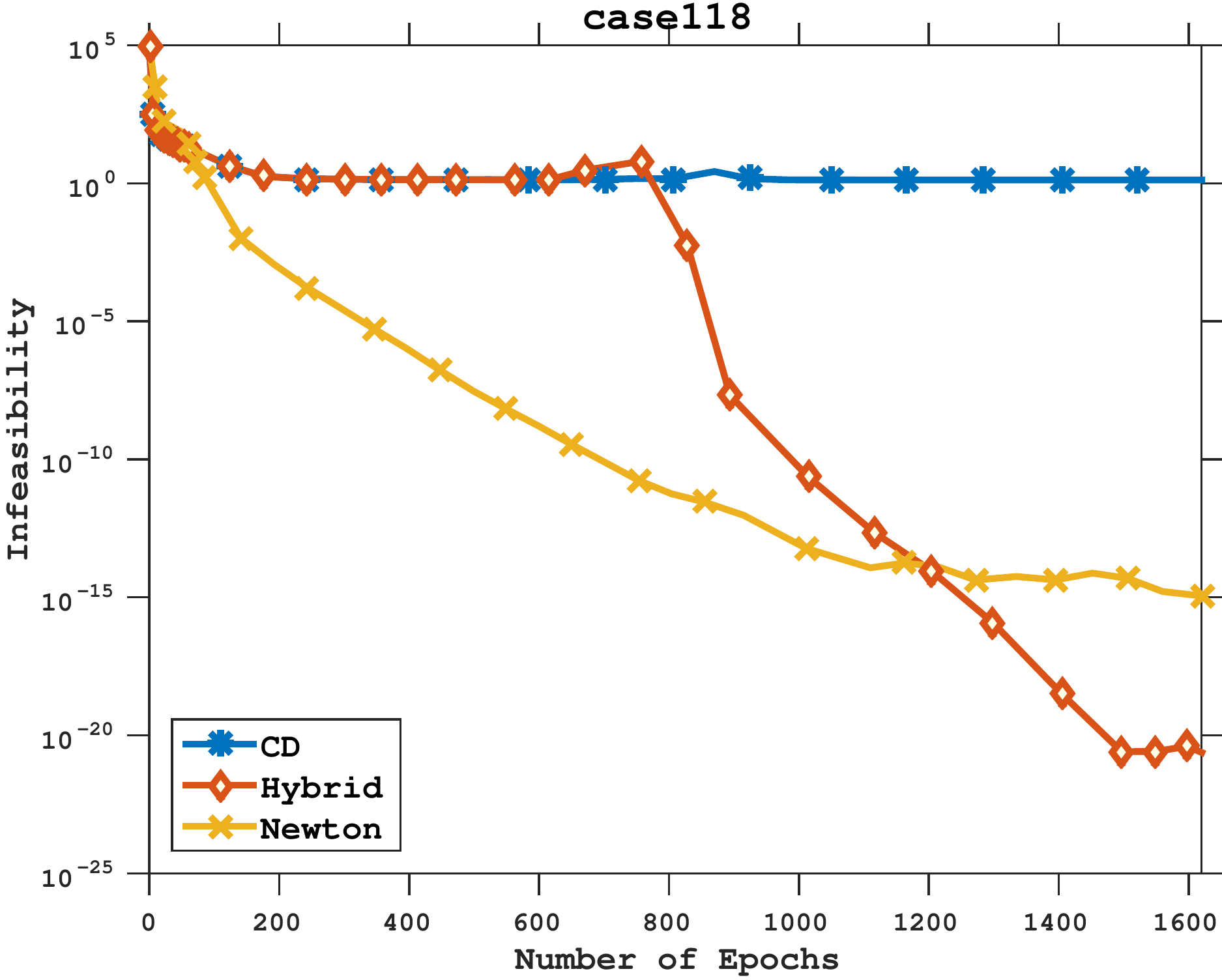}\\
\includegraphics[scale=0.2]{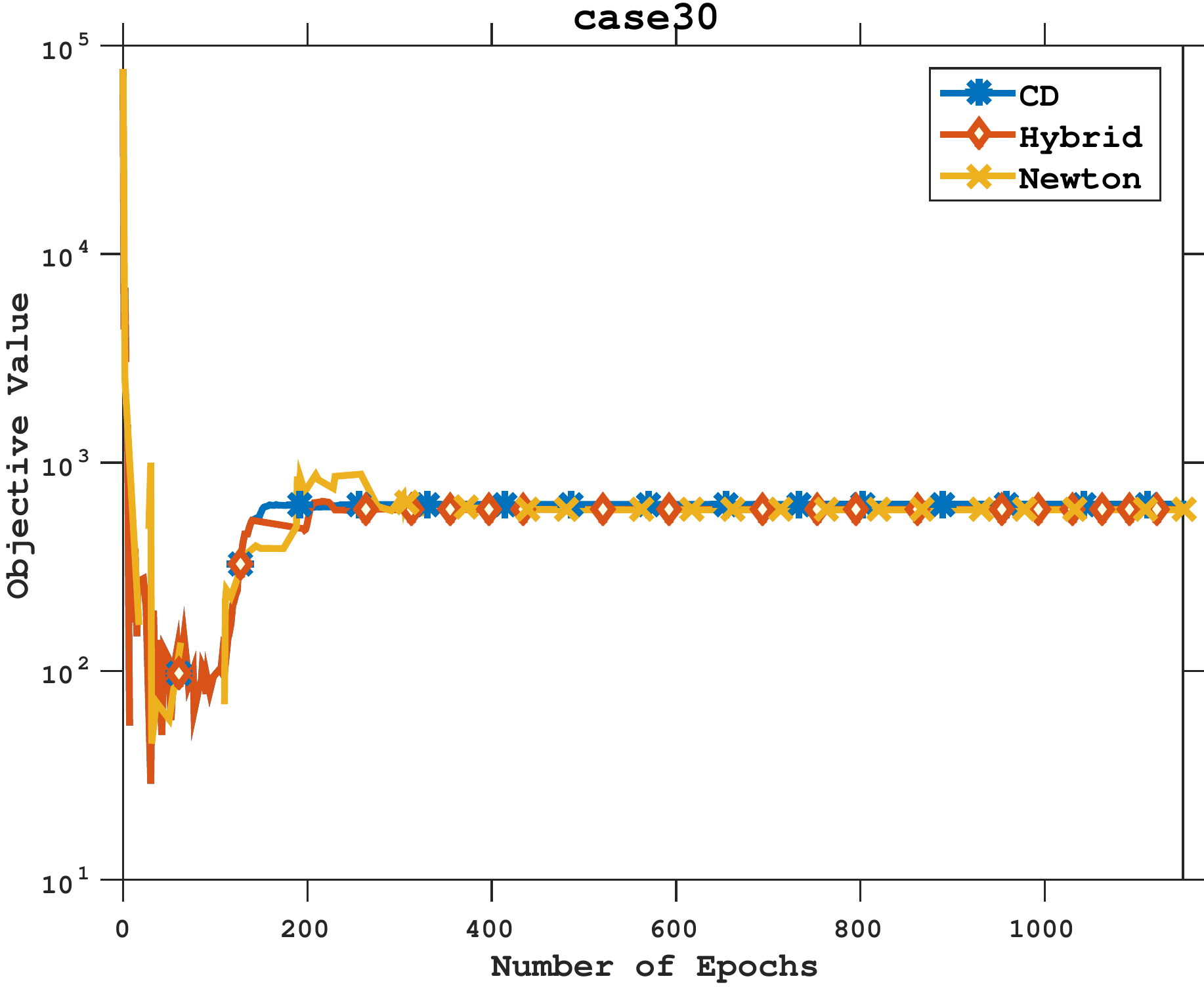}
\includegraphics[scale=0.2]{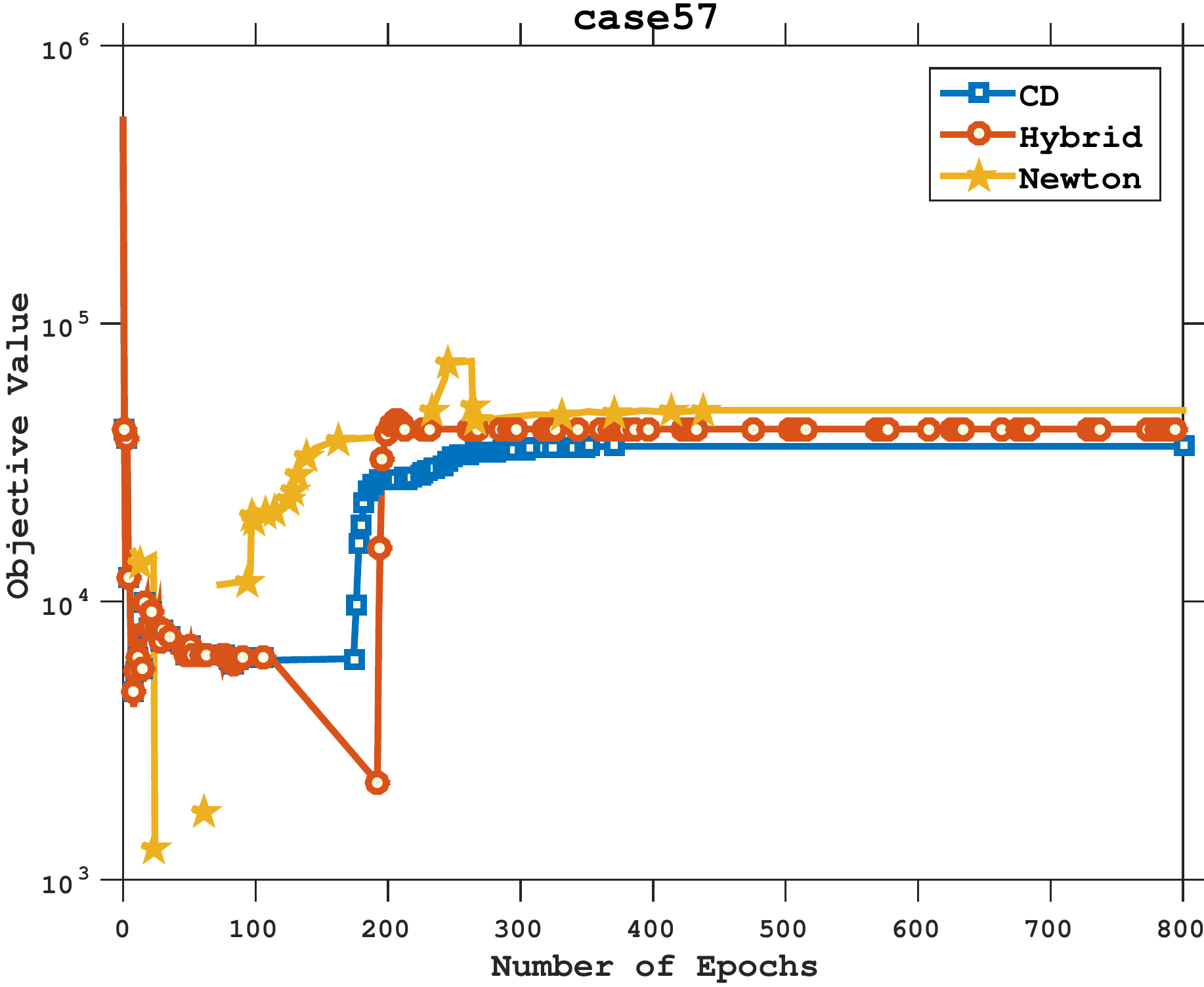}
\includegraphics[scale=0.2]{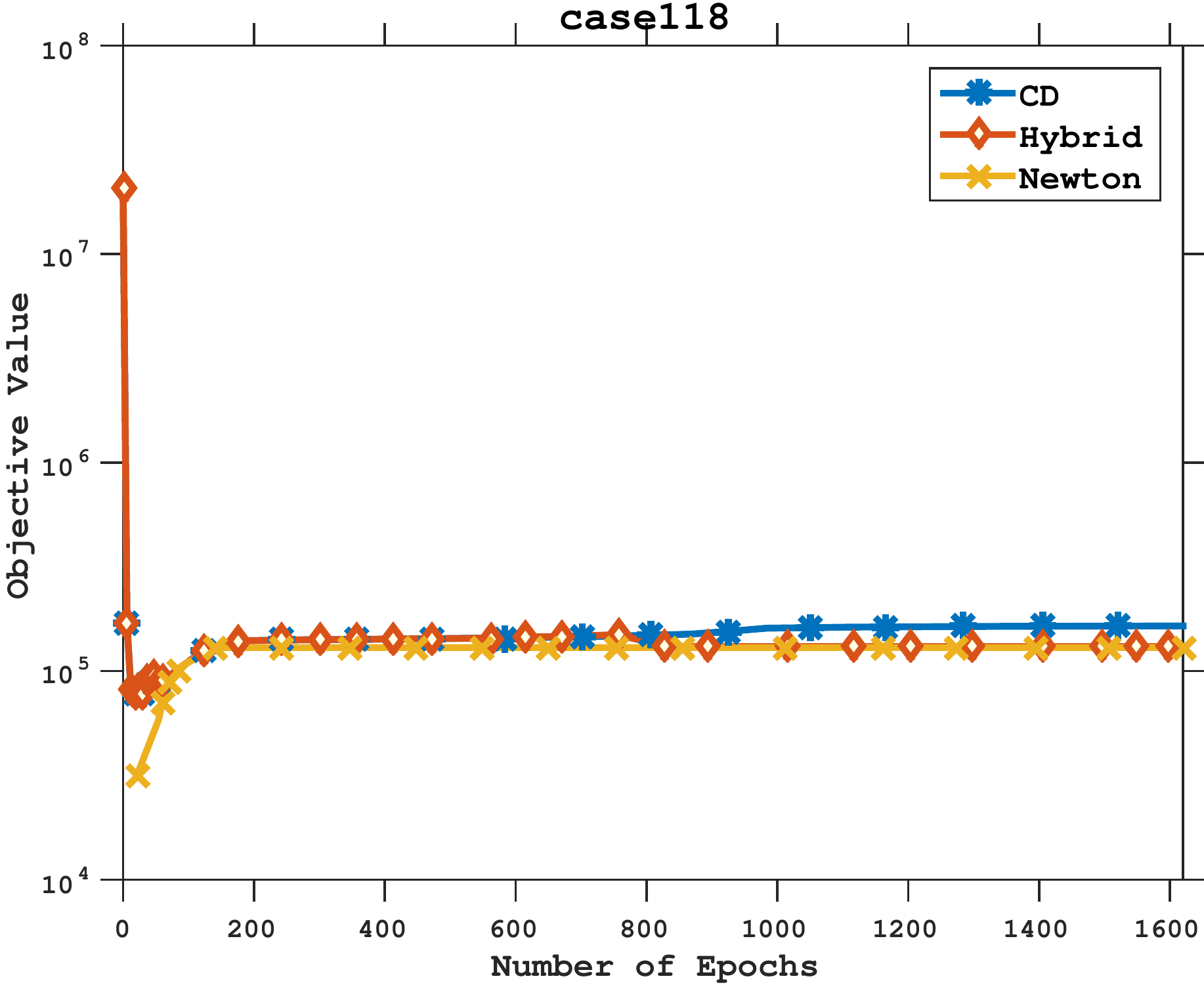}\\
 \caption{The performance of the hybrid method on three IEEE test systems: 30-bus, 57-bus, and 118-bus.}
 \label{fig:additional}
\end{figure*}

\begin{figure*}[t]
\includegraphics[scale=0.2]{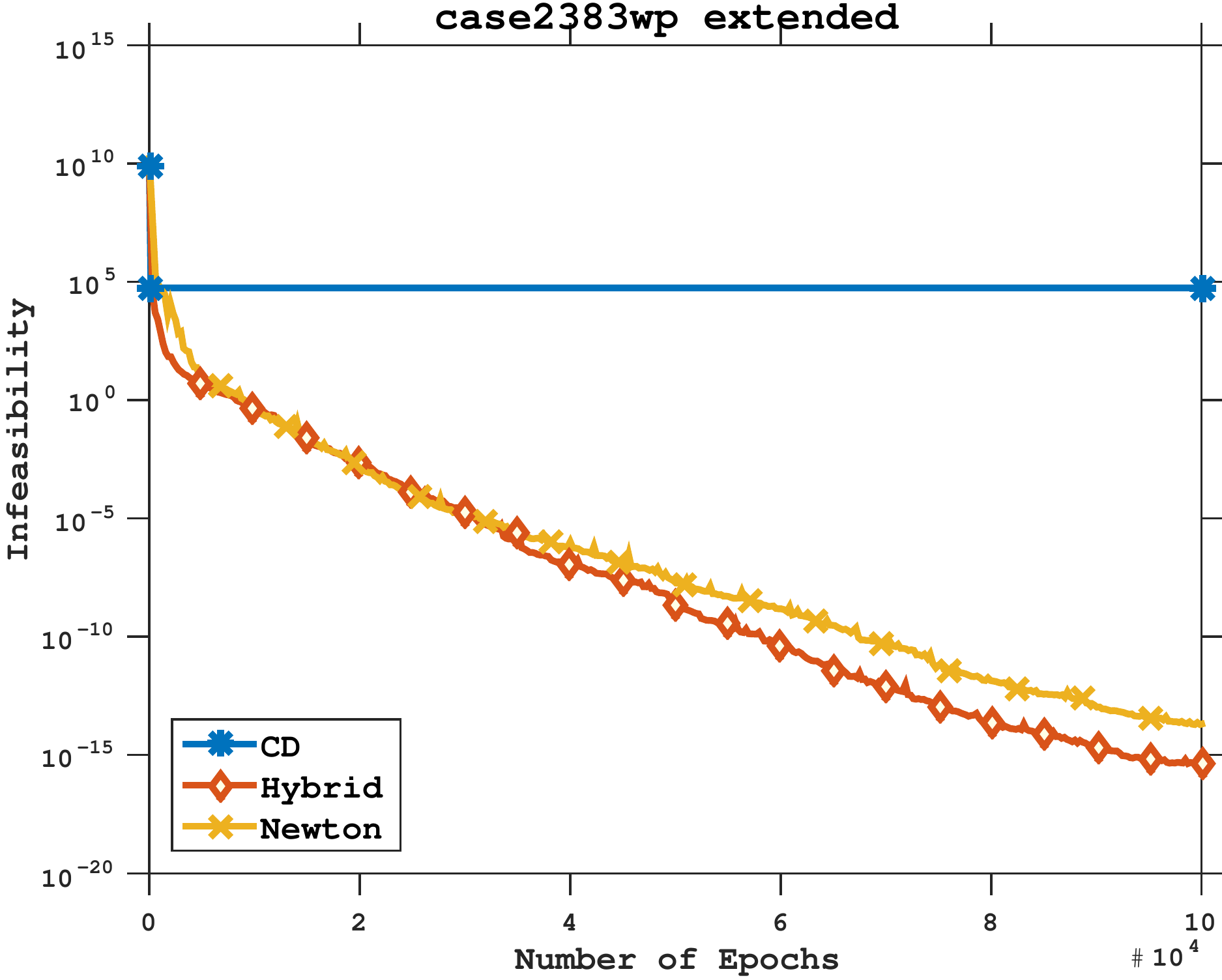}
\includegraphics[scale=0.2]{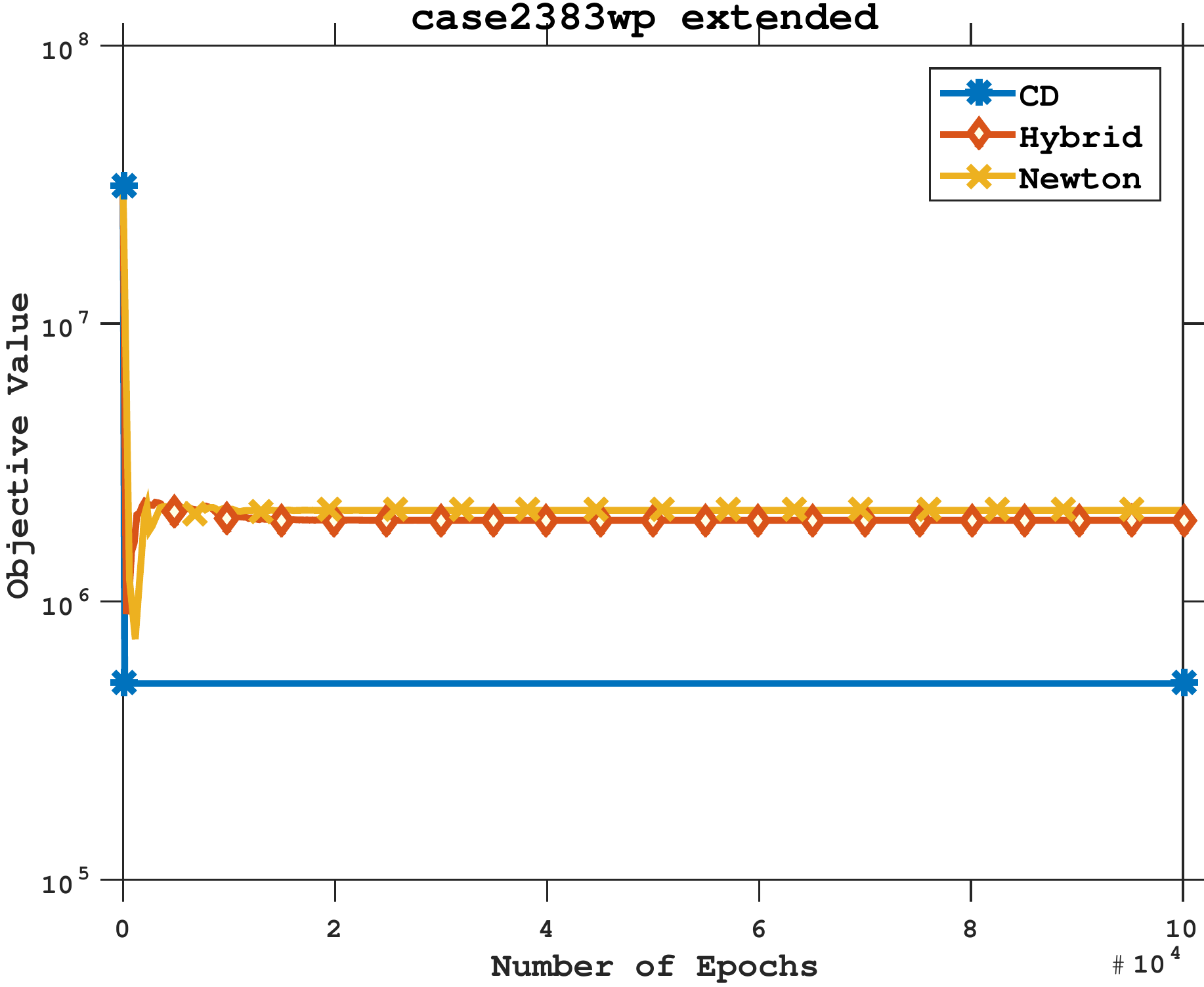}
\includegraphics[scale=0.2]{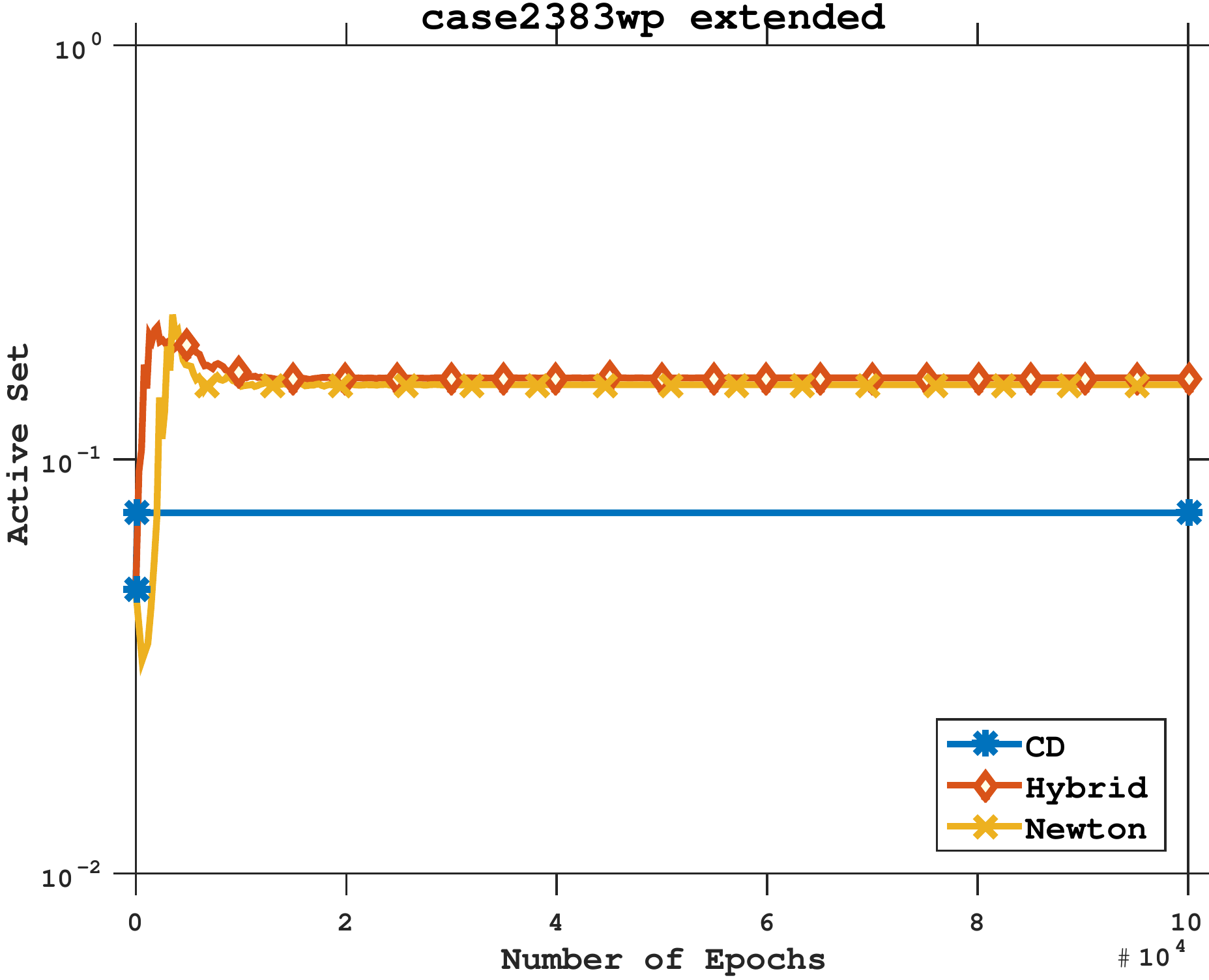}
%\\ $ $\\
%\includegraphics[scale=0.2]{newimg/case2383wpInfeasTime1}
%\includegraphics[scale=0.2]{newimg/case2383wpObjTime1}
%\includegraphics[scale=0.2]{newimg/case2383wpActsetTime1}
 \caption{The performance of the hybrid method on a snapshot of the Polish transmission system (case2383wp).}
 \label{fig:case2383wp} 
\end{figure*}

 Some evidence of the need for such a test %as to when it is safe to switch over to the Newton method,
 is provided in Figure \ref{fig:motivation}.
 There, the plot the evolution of infeasibility (top row) and objective function (bottom row) over the number of epochs, 
 while varying the number of epochs after we switch from solving the convexification 
 to the Newton method. 
 Even after a number of iterations, each of which decreases the value of the Lagrangian, 
 the Newton method can diverge.
 See, for example, the series denoted CD 2 in the middle plot in the top row, where 
 one switches-over after 2 epochs of the coordinate descent on the convexification for the IEEE 118-bus test system, 
 but where the infeasibility does not seem to fall below 1, ever.
 In the middle plot in the bottom row, one may also observe the variety of local optima reached,
 by switching after 32 epochs (CD 32), 128 epochs (CD 128), and 500 epochs (CD 500) on 
 the same instance.
 %For additional plots, we refer to Appendix \ref{app:motivational}.
 
 %\subsection{The Results}

To validate the impact of our approach, we performed numerical experiments on a collection of well-known instances \cite{Matpower}.
The experiments have been performed on a computer with an Intel Xeon CPU E5-2620 clocked at 2.40GHz and 128 GB of RAM.
We have used randomisation in generating the initial point, as well as in the sampling of coordinates, but 
we have used a fixed random seed for all runs of all methods. 
 %the same initial point as in Matpower \cite{Matpower}: 
 %voltage magnitudes set uniformly to 1,
 %phase angles set uniformly to 0,
 %and both active and reactive powers generated set to the mid-point of the 
 %feasible intervals at the respective generators.
Throughout, we compare the performance of the 
 coordinate descent of \cite{MarecekTakac2015} 
on the Lavaei-Low SDP relaxation \cite{lavaei2012zero} (plotted in blue),
against the Newton method on the non-convex Lagrangian (plotted in yellow),
against the performance of a variant of the hybrid method (plotted in red), 
 which switches from the coordinate descent on the convexification to 
 to Newton method on the non-convex Lagrangian, when the $\alpha$-$\beta$ test
is satisfied.
In Figures~\ref{fig:case300}--\ref{fig:case2383wp},
 we present a sample of the results.

For the first illustration, we chose the IEEE 300-bus test system, 
  and plotted both a measure of infeasibility (left; see \cite{MarecekTakac2015} for a definition) and 
  the objective function value (middle)
  against both wall-clock time (top row) and epochs (bottom row) in Figure~\ref{fig:case300}.
We use the the term epoch to mean $m$
  iterations of coordinate descent,
  or $m$ coordinate-wise Newton steps,
  for an instance in dimension $m$.
As can be seen by comparing the top and bottom row, 
  the wall-clock time corresponding to one epoch across the three methods is similar.
On the other hand, the convergence rates are visibly different, with the infeasibility decreasing at a quadratic rate
for the Newton and hybrid method. 

Further, we present the results on three more IEEE test systems in Figure \ref{fig:additional}
 in a more concise form with only the evolution of infeasibility (top row) and objective function (bottom row)
 over the number of epochs. 
The 30-bus (on the left) and 118-bus (on the right) test systems illustrate the typical performance:
the evolution of infeasibility of the hybrid method overlaps with the first-order method
  until the switch-over. Henceforth, the quadratic rate of convergence
  resembles that of the Newton method, except with a better starting point. 
The 57-bus test system (in the middle) demonstrates the importance of the starting point: 
  our implementation of the Newton method from the Matpower starting point does not converge. 

Next, to illustrate the scalability of the approach, we present the results on a snapshot of the Polish system
 in Figure \ref{fig:case2383wp}.
There are 2383 buses in the snapshot, and more importantly, 
  tap-changing and phase-shifting transformers, 
  double-circuit transmission lines, and multiple generators at each bus,
  which complicate the formulation of the thermal limits,
  as explained in Section 5.2 of \cite{MarecekTakac2015}. 
Despite the preliminary nature of our implementation, compared to the established codes, developed over a decade or more  \cite{Matpower},
 the convergence seems very robust. 
%to infeasibility $10^{-15}$ 

Finally, in the right-most plots of Figures~\ref{fig:case300} and \ref{fig:case2383wp},
  we plot the ratio of the cardinality of the active set to the number of inequalities 
  over the epochs or time. 
This provides an empirical justification for the choice of Algorithm \ref{alg:HybridActiveset}:
  the active set clearly stabilises much earlier than the objective function value,
  and is only a small fraction of the count of the ineqalities,
  which allows for the short run-time of the test implementing Proposition \ref{mubound}.

\section{Related Work}

Let us present a brief overview of the rich history of the study of the convergence of Newton method.
The best known result in the field is the theorem of Kantorovich \cite{kantorovich1948},
which formalises the assumptions under which whenever for a closed ball of radius $t_*$ centered at $x_0$, it holds that
$\nabla F(x) + \nabla F(x)^T \succ 0$ for all $x$ in the ball, 
the ball is a domain of monotonicity for the function $F$.
(See also Appendix \ref{sec:Kantorovich}.)
Traditionally, it has been assumed that testing the property across the 
closed ball is difficult. 

Recently, Henrion and Korda \cite{6606873} have shown that the domain of monotonicity of a polynomial system can be 
computed by solving an infinite-dimensional linear program over the space of 
measures, whose value can be approximated by a certain hierarchy of convex semidefinite optimisation problems.
See also the work of Valm{\' o}rbida et al.\ \cite{5399969,6859263,7268848} % and others . % 
in the context of partial differential equations, and elsewhere \cite{7106722}.
Dvijotham et al. \cite{Dj15,Dj15b} showed that it can also be be cast as a 
certain non-convex semidefinite optimisation problem.
Notice, however, that this line of work \cite{Dj15,Dj15b,6606873}
does not consider inequalities and may be rather computationally challenging.
Similarly, the $\alpha$-$\beta$ theory \cite{ShubSmale1993,chen1994approximate,Cucker1999},
does not consider inequalities.

To summarise: traditionally, the convergence of Newton method could be 
guaranteed only by the non-constructive arguments of the theorem of Kantorovich.
Alternatively, one could the recently developed approaches \cite{6606873,Dj15,Dj15b}, 
albeit at a computational cost possibly higher than that of solving ACOPF.
%not least in terms of the computational cost of solving the problems involved.
%, as solving one instance of (\ref{EqDjA}--\ref{EqDjF}) is NP-Hard, just as solving the original ACOPF is NP-Hard.
Our approach seems to improve upon these considerably.

%\todo[inline]{Is it good if we include a short paragraph about parallelism of the methodology? I remember that in the email, Lavaei suggested about possibility of distributed fashion to accommodate the topic of the special issue.}

\section{Conclusions}

Without the use of (hierarchies of) convex relaxations, Newton-type methods can converge to particularly bad local optima of non-convex problems. 
Even the fastest first-order methods for computing strong convex relaxations are, however, rather slow on their own.
Hybrid methods combining first-order methods for the strong convex relaxations and Newton-type methods for the non-convex problems 
combine the guarantees of convergence associated with (hierarchies of) convex relaxations 
and the quadratic rates of convergence of Newton method.
Crucially, such hybrid methods can be implemented in a distributed fashion, as discussed in \cite{MarecekTakac2015}.
This improves upon \cite{Dj15,Dj15b} and opens up many novel directions
for future research.

\newpage
\bibliographystyle{abbrv} 
\bibliography{acopf,literature,jie}

%newpage
\appendix
\section{Kantorovich's theorem}
\label{sec:Kantorovich}

Let $X$ be a Banach space, $X^*$ its dual, $F$ a functional on $X$, 
and $\nabla F$ a Fr{\' e}chet derivative of $F$. 
Further, define the open and closed balls at $x\in X$ as:
\begin{subequations}
\begin{align}
B(x,r) & \coloneq  \{ y\in X ;\; \|x-y\|<r \} \textrm{ and } \\
B[x,r] & \coloneq  \{ y\in X ;\; \|x-y\|\leqslant r\},
\end{align}
\end{subequations}
respectively. 
Kantorovich's theorem can be stated as follows:

\begin{theorem}[\cite{kantorovich1948}] \label{th:knclass}
  Let $X$, $\banachb$ be Banach spaces, $C\subseteq X$
  and $F:{C}\mapsto \banachb$ a continuous function, continuously
  differentiable on $\mathrm{int}(C)$. Take $x_0\in \mathrm{int}(C)$,
  $L,\, b>0$ and suppose that
  \begin{description}
  \item [1)] $\nabla F(x_0)$ is non-singular,
  \item  [2)]
    \( \|
    \nabla F(x_0)^{-1}\left[ \nabla F(y)- \nabla F(x)\right]
    \| \leq L\|x-y\|
    \)\;\;  for any $x,y\in C$,
  \item [3)] $ \| \nabla F(x_0)^{-1}F(x_0)\|\leq b$,
  \item [4)] $2bL\leq 1$.
  \end{description}
  If $B[x_0,t_*]\subset C$, then the sequences $\{x_k\}$ generated by Newton method for
  solving $F(x)=0$ with starting point $x_0$,
  \begin{equation} \label{ns.KT} 
    x_{k+1} \coloneq {x_k}- [\nabla F(x_n)]^{-1}F(x_k), \qquad k=0,1,\cdots, 
  \end{equation}
  is well defined, is contained in $B(x_0, t_*)$, converges to a
  point $x_*\in B[x_0,t_*]$, which is the unique zero of $F$ in
  $B[x_0,t_*]$, and 
  \begin{equation}
    \label{eq:q.conv.x}
    \|x_*-x_{k+1}\|\leq \frac{1}{2} \|x_*-x_k \|, \qquad
    k=0,1,\,\cdots. 
  \end{equation}
   Moreover, if $2bL<1$, then
   \begin{align} \label{eq:q2}
     \|x_*-x_{k+1}\| & \leq\frac{1-\theta^{2^k}}{1+\theta^{2^k}}
     \frac{ L}{2\sqrt{1-2bL}}\|x_*-x_k\|^2 \\ & \leq \nonumber
     \frac{ L}{2\sqrt{1-2bL}}\|x_*-x_k\|^2, \quad k=0,1,\cdots, 
   \end{align}
  where $\theta \coloneq  t_*/t_{**}<1$, and $x_*$ is the
  unique zero of $F$ in $B[x_0,\rho]$ for any $\rho$ such that
  \begin{equation}
 t_*\leq\rho<t_{**},\qquad B[x_0,\rho]\subset C,
  \end{equation}
where
  \begin{equation}
    \label{eq:scroots}
    t_* \coloneq  \frac{1-\sqrt{1-2bL}}{L},\qquad
    t_{**} \coloneq \frac{1+\sqrt{1-2bL}}{L}. 
  \end{equation}
\end{theorem}
\noindent

\end{document}